\documentclass[12pt,draftcls,onecolumn]{IEEEtran}

\usepackage{lipsum}
\usepackage{cite}
\usepackage{graphicx}      
\usepackage{amsfonts,amssymb,amstext,amsthm,verbatim,amsmath,graphicx,color,enumerate}
\usepackage{bbold}

\usepackage{tikz}
\usepackage{pgfplots}
\usepackage{float}

\usepackage{algorithm,algorithmic,xspace}

\usepackage{epstopdf}

\usepackage{bm}

\usetikzlibrary{arrows.meta}
\usetikzlibrary{positioning, calc}
\tikzset{
 bus/.style={draw, circle, minimum size=2em,inner sep=0pt},
 busg/.style={draw, circle, minimum size=2em,inner sep=0pt,color=black}
}
\usetikzlibrary{positioning,backgrounds}

\graphicspath{{figs/}}

\newtheorem{theorem}{Theorem}[section]
\newtheorem{proposition}[theorem]{Proposition}
\newtheorem{corollary}[theorem]{Corollary}
\newtheorem{definition}[theorem]{Definition}
\newtheorem{lemma}[theorem]{Lemma}

\newtheorem{assumption}[theorem]{Assumption}

\newcommand{\idxi}{i}
\newcommand{\idxj}{j}
\newcommand{\idxl}{\ell}
\newcommand{\idxm}{m}
\newcommand{\idxh}{h}
\newcommand{\idxt}{t}

\newcommand{\graph}{G}
\newcommand{\nodes}{\{1,\dots,N\}}
\newcommand{\edges}{E}

\newcommand{\nNodes}{N}
\newcommand{\nEdges}{n}
\newcommand{\inter}{S}
\newcommand{\comm}{{\texttt{cmm}}}

\newcommand{\Prob}{\mathbb{P}}
\newcommand{\state}{\mathcal{C}}
\newcommand{\nState}{C}
\newcommand{\pState}{X}
\newcommand{\oState}{x}
\newcommand{\vState}{c}

\newcommand{\test}{\mathcal{R}}
\newcommand{\nTest}{R}
\newcommand{\pTest}{Y}
\newcommand{\oTest}{y}
\newcommand{\boTest}{\bm{y}}
\newcommand{\bpTest}{\bm{Y}}
\newcommand{\vTest}{r}

\newcommand{\ro}{\gamma} 
\newcommand{\bro}{\bm{\gamma}} 
\newcommand{\Ro}{\Gamma} 
\newcommand{\tensor}{\theta}
\newcommand{\Tensor}{\Theta}
\newcommand{\btensor}{\bm{\theta}}

\newcommand{\IO}{\leftrightarrow} 
\newcommand{\ItoO}{\rightarrow} 
\newcommand{\OtoI}{\leftarrow} 

\newcommand{\LF}{L}

\allowdisplaybreaks

\def \jphml/{JPH-ML}
\def \djphml/{DJPH-ML}
\def \djphfr/{DJPH-FR}

\def \JPHNR/{JPH-NR}
\def \JPHFR/{JPH-FR}

\DeclareMathOperator*{\argmin}{argmin}
\DeclareMathOperator*{\argmax}{argmax}

\title{Interaction-Based Distributed Learning\\ in
  Cyber-Physical and Social Networks} 

\author{Francesco Sasso, Angelo Coluccia, and Giuseppe
  Notarstefano \thanks{Francesco Sasso, Angelo Coluccia and Giuseppe Notarstefano are with the Department of
Engineering, Universit\`a del Salento, via Monteroni, 73100, Lecce, Italy, \{name.lastname\}@unisalento.it. 

A very preliminary version of this work appeared as \cite{coluccia2014hierarchical} in which the high-level idea of the problem set-up was proposed for a very simplified case in which both the state and the score are binary. This result is
part of a project that has received funding from the European Research Council
(ERC) under the European Unions Horizon 2020 research and innovation
programme (grant agreement No 638992 - OPT4SMART).
}  }

\begin{document}
\bstctlcite{IEEEexample:BSTcontrol}

\maketitle

\begin{abstract}
In this paper we consider a network scenario in which agents can evaluate each
other according to a score graph that models some physical or social
interaction. The goal is to design a distributed protocol, run by the agents,
allowing them to learn their unknown state among a finite set of possible
values. We propose a Bayesian framework in which scores and states are
associated to probabilistic events with unknown parameters and hyperparameters
respectively. We prove that each agent can learn its state by means of a local
Bayesian classifier and a (centralized) Maximum-Likelihood (ML) estimator of the
parameter-hyperparameter that combines plain ML and Empirical Bayes
approaches. By using tools from graphical models, which allow us to gain insight
on conditional dependences of scores and states, we provide two relaxed
probabilistic models that ultimately lead to ML parameter-hyperparameter
estimators amenable to distributed computation. In order to highlight the
appropriateness of the proposed relaxations, we demonstrate the distributed
estimators on a machine-to-machine testing set-up for anomaly detection and on a
social interaction set-up for user profiling.	
\end{abstract}


\section{Introduction}
A common feature of modern cyber-physical and social networks is the capability
of the subsystems of interacting locally with the possibility of testing,
monitoring or simply rating the neighboring subsystems. In social networks
individuals continuously interact among themselves and (more and more often)
with cyber members by sharing contents and expressing opinions or ratings on
different topics. Similarly, in industrial (control) networks, as
power-networks, smart grids or automated factories, devices have the possibility
to test each other for physical diagnosis or to get an indication of the level
of trust, in order to prevent catastrophic faults or malware attacks.
In this paper we model a general network scenario in which nodes can mutually
rate, i.e., can give/receive a score to/from other ``neighboring'' nodes, and
aim at deciding their own (or their neighbors') state. The state may indicate a
social orientation, influencing level, or the belonging to a thematic community,
or it may characterize the level of faultiness or (mis)trust.
Due to the large-scale nature of these complex systems, centralized solutions to
classify nodes exhibit limitations both in terms of computation burden and
privacy preserving, so that distributed solutions need to be investigated.

\textit{Literature review:}
In the past few years, a great interest has been devoted to distributed estimation
schemes in which nodes aim at agreeing on a common parameter, e.g., by means of
Maximum Likelihood (ML) approaches,
\cite{barbarossa2007decentralized,schizas2008consensus,chiuso2011gossip}.
In \cite{coluccia2013distributed,coluccia2016bayesian} a more general Bayesian framework is considered,
in which nodes estimate local parameters, rather than reaching a consensus on a
common one. The estimation of the local parameters is performed by resorting to
an Empirical Bayes approach in which the parameters of the prior distribution,
called hyperparameters, are estimated through a distributed algorithm. The
estimated hyperparameters are then combined with local measurements to obtain
the Minimum Mean Square Error estimator of the local parameters.
Consensus-based algorithms have been proposed in
\cite{chiuso2011gossip,fagnani2014distributed} for the simultaneous distributed
estimation and classification of network nodes.
In \cite{montijano2015distributed} trustworthy consensus is studied, which is
able to cope with data association mistakes and measurement outliers. To this
aim, different hypotheses are generated and voted for, and nodes can change
their opinion according to a dynamic voting process.
More generally, several hypothesis testing problems have received attention in network
contexts. Differently from our set-up, these references consider a scenario in
which agents aim at learning a \emph{common} unobservable state.
In \cite{dandach2012accuracy} a group of individuals needs to decide on two
alternative hypotheses; the global decision is, however, taken by a fusion
center collecting local decisions. In the recent literature on distributed
social learning \cite{jadbabaie2012non} agents aim at learning a \emph{common}
unobservable state of the world in a finite set of possibilities, by making
repeated noisy observations. A challenge addressed in this area is the design
and analysis of non-Bayesian learning schemes in which each agent processes its
own and its neighbors' beliefs
\cite{shahrampour2013exponentially,lalitha2014social}.
More recent references investigate the effect of network size/structure and link
failures \cite{shahrampour2016distributed} or the presence of faulty nodes
\cite{su2016non} on the efficiency of the learning rules.
In \cite{molavi2016foundations} the authors consider a class of non-Bayesian
learning rules (in which agents treat the neighboring beliefs as sufficient
statistics) that are shown to take a log-linear form. It is also shown how the
long-run beliefs depend on the structure of the local rule and on the interaction
with neighbors.
In \cite{nedic2017fast} non-Bayesian learning protocols for time-varying
networks with possibly conflicting hypotheses are proposed.
An overview of recent results on distributed learning algorithms with their
convergence rates is provided in \cite{nedic2016tutorial}.
A different batch of references investigates discrete-time or continuous-time
dynamic laws describing interpersonal influences in groups of individuals and
investigate the emerging of asymptotic opinions. 
These dynamics are proven to
lead to single or multiple opinions depending on the network parameters and/or
initial conditions
\cite{mirtabatabaei2012opinion,mirtabatabaei2014reflected,friedkin2016network}.
The tutorial \cite{frasca2015distributed} reviews opinion formation in social
networks and other applications by means of randomized distributed algorithms.
The problem of self-rating in a social environment is discussed in
\cite{li2016self}, where agents can perform a predefined task, but with
different abilities. The paper presents a distributed dynamics allowing each
agent to self-rate its level of expertise/performance at the task, as a
consequence of pairwise interactions with the peers.

\textit{Statement of contributions:} The contributions of this paper are as
follows. First, we set up a learning problem in a network context in which each
node needs to classify its own local state rather than a common variable of the
surrounding world. Moreover, nodes learn their state based on observations
coming from the interaction with other nodes, rather than on measurements
collected locally from the surrounding environment. Interactions among nodes are
expressed by evaluations that a node performs on other ones, modeled through a
weighted digraph that we call score graph. This general scenario captures a wide
variety of contexts arising from social relationships as well as
machine-to-machine interactions. Motivated by these contexts, in the proposed
set-up nodes are assumed to have only a partial knowledge of the
world. Specifically, we devise a Bayesian probabilistic framework wherein,
however, both the parameters of the observation model and the hyperparameters of
the prior distribution are allowed to be unknown. In this sense, this framework
can be seen as an Empirical Bayes approach with additional unknown parameters in
the conditional distribution of the observables.
Second, in order to solve this interaction-based learning problem, we propose a
learning approach combining a local Bayesian classifier with a joint
parameter-hyperparameter Maximum Likelihood estimation approach. 
For the local Bayesian classifier, we derive a closed form expression depending
only on aggregated evaluations from the neighbors. This expression can be used
to obtain both the Maximum A Posteriori (MAP) decision as well as a ranking of
the alternatives with associated (probabilistic) trust.
We show that the ML estimator of the global parameter-hyperparameter is not
amenable to distributed computation, and is computationally intractable even for
moderately small networks. To overcome this main issue, we explore the
probabilistic structure of the problem by resorting to the conceptual tool of
graphical models,\cite{koller2009probabilistic}, to efficiently describe the
conditional dependencies. In doing so, we identify two reasonable relaxations
that lead to modified likelihood functions exhibiting a separable structure. The
resulting optimization problems for the parameter-hyperparameter estimation are,
thus, amenable to distributed computation. In particular, we propose a
node-based relaxation, for which available distributed optimization algorithms
can be used, and a full relaxation for which we propose an ad-hoc distributed
algorithm combining a local descent step with a diffusion (consensus-based)
step. 
We validate the performances of the proposed distributed estimators through
Monte Carlo simulations on two interesting scenarios, namely on anomaly
detection in cyber-physical networks and user profiling in social networks.

\textit{Organization:} In Section~\ref{sec:setup} we set-up a Bayesian framework
for interaction-based learning problems by means of a suitable graphical model,
and introduce two scenarios of interest in cyber-physical and social
networks. In Section~\ref{sec:algorithm} we derive the proposed distributed
classification algorithms based on the local Bayesian classifier and the
distributed parameter-hyperparameter estimators obtained through the two ML
relaxations. Finally, in Section~\ref{sec:simulations} we assess the
performances of the proposed schemes by means of a numerical (Monte Carlo)
analysis for the application scenarios introduced in Section~\ref{sec:setup}.

\section{Bayesian framework for interaction-based learning}\label{sec:setup}
In this section, we set up the interaction-based learning problem in which
agents of a network interact with each others according to a score graph. To
learn its own state (or a neighbor's state) each node can use observations
associated to incoming or outcoming edges. We propose a Bayesian probabilistic
model with unknown parameters, which need to be estimated to solve
the learning problem.

\subsection{Interaction network model}
We consider a \emph{network of agents} able to perform evaluations of other
agents. The result of each evaluation is a score given by the evaluating agent
on the evaluated one. Such an interaction is described by a \emph{score
  graph}. Formally, we let $\nodes$ be the set of agent identifiers and
$\graph_\inter = (\nodes,\edges_\inter)$ a digraph such that
$(\idxi,\idxj)\in \edges_\inter$ if agent $\idxi$ evaluates agent $\idxj$.  We
denote by $\nEdges$ the total number of edges in the graph, and assume that each
node has at least one incoming edge in the score graph, that is there is at
least one agent evaluating it.

We let $\state$ and $\test$ be the set of possible state and score values,
respectively. Being finite sets, we can assume
$\state = \{\vState_1,\dots,\vState_\nState\}$ and
$\test = \{\vTest_1,\dots,\vTest_\nTest\}$, where $\nState$ and $\nTest$ are the
cardinality of the two sets, respectively. Consistently, in the network we
consider the following quantities:
\begin{itemize}
\item $\oState_\idxi \in \state$, unobservable \emph{state} (or community) of agent $\idxi$;
\item $\oTest_{\idxi\idxj} \in \test$, \emph{score} (or evaluation result) of the evaluation performed by
  agent $\idxi$ on agent $\idxj$.
\end{itemize}
An example of score graph with associated state and score values is shown in
Fig.~\ref{fig:score-graph-example}.
\begin{figure}[!htpb]
\centering
\includegraphics[scale=0.8]{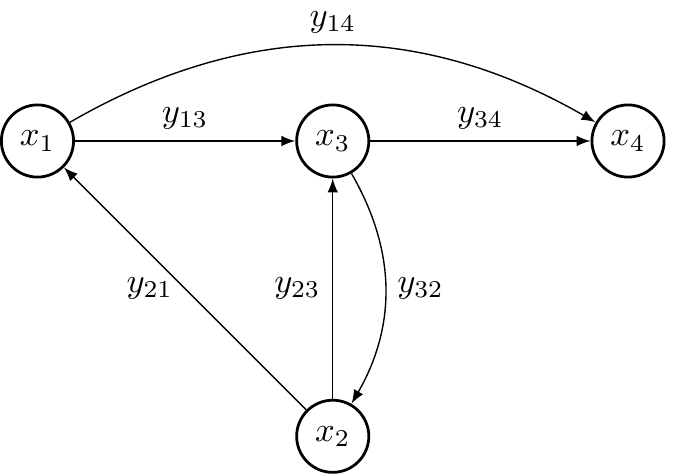}
\caption{Example of a score graph $\graph_\inter$.}
\label{fig:score-graph-example}%
\end{figure}

Besides the evaluation capability, the agents have also \emph{communication} and
\emph{computation} functionalities. That is, agents communicate according to a
time-dependent directed \emph{communication graph}
$\idxt\mapsto\graph_\comm(\idxt) = (\nodes, \edges_\comm(\idxt))$, where the
edge set $\edges_\comm(\idxt)$ describes the communication among agents:
$(\idxi,\idxj)\in\edges_\comm(\idxt)$ if agent $\idxi$ communicates to $\idxj$
at time $\idxt\in\mathbb{Z}_{\geq 0}$. We introduce the notation
$\nNodes_{\comm,\idxi}^{I}(\idxt)$ and $\nNodes_{\comm,\idxi}^{O}(\idxt)$ for the
in- and out-neighborhoods of node $\idxi$ at time $\idxt$ in the communication
graph. We will require these neighborhoods to include the node $\idxi$ itself;
formally, we have
\begin{align*}
\nNodes_{\comm,\idxi}^{I}(\idxt) &= \{\idxj : (\idxj,\idxi)\in\edges_\comm(\idxt)\}\cup\{\idxi\},\\
\nNodes_{\comm,\idxi}^{O}(\idxt) &= \{\idxj : (\idxi,\idxj)\in\edges_\comm(\idxt)\}\cup\{\idxi\}
\end{align*}

We make the following assumption
on the communication graph:

\begin{assumption}\label{ass:connectivity}
\! There exists an integer $Q\ge1$ such that the graph
  $\bigcup_{\tau=\idxt Q}^{(\idxt+1)Q-1} \!\!\graph_\comm(\tau)\!$ is strongly connected $\forall \, t\!\ge\!0$.
\end{assumption}

We point out that in general the (time-dependent) communication graph, modeling
the distributed computation, is not necessarily related to the (fixed) score
graph. We just assume that when the distributed algorithm starts each
node $i$ knows the scores received by in-neighbors in the score graph. This
could be obtained by assuming that if the distributed algorithm starts at some
time $t_0$, then for some $\bar{t}>0$, $\graph_\inter \subseteq \bigcup_{\tau= t_0 - \bar{t}}^{t_0}
\graph_\comm(\tau)$. 

\subsection{Bayesian probabilistic model}
We consider the score $\oTest_{\idxi\idxj}, (\idxi,\idxj)\in\edges_\inter$, as
the (observed) realization of a random variable denoted by
$\pTest_{\idxi\idxj}$; likewise, each state value
$\oState_{\idxi}, \idxi\in\nodes$, is the (unobserved) realization of a random
variable $\pState_\idxi$.  In order to highlight the conditional dependencies
among the random variables involved in the score graph, we resort to the tool of
graphical models and in particular of Bayesian networks,
\cite{koller2009probabilistic}. Specifically, we introduce the \emph{Score
  Bayesian Network} with $\nNodes+\nEdges$ nodes $\pState_\idxi$,
$\idxi=1,\dots,N$, and $\pTest_{\idxi\idxj}$, $(\idxi,\idxj)\in\edges_\inter$ and
$2\nEdges$ (conditional dependency) arrows defined as follows. For each
$(\idxi,\idxj)\in\edges_\inter$, we have
$\pState_\idxi\rightarrow\pTest_{\idxi\idxj}\leftarrow\pState_\idxj$ indicating
that $\pTest_{\idxi\idxj}$ conditionally depends on $\pState_\idxi$ and
$\pState_\idxj$.  In Fig.~\ref{fig:graphical-model-example} we represent the
Score Bayesian Network related to the score graph in
Fig.~\ref{fig:score-graph-example}.
\begin{figure}[!htpb]
\centering
\includegraphics[scale=0.8]{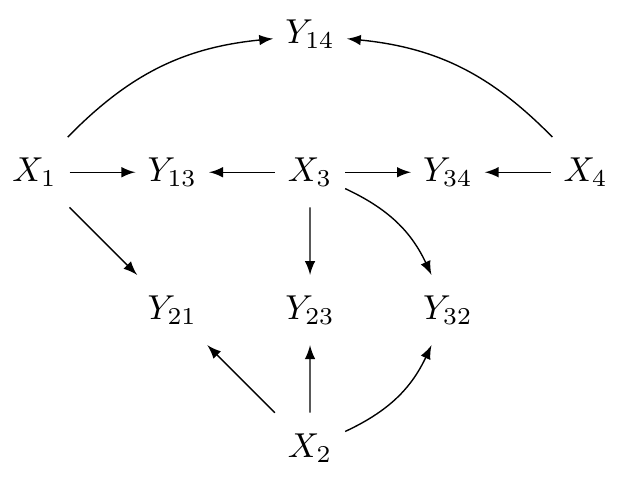}
\caption{The score Bayesian network related to the score graph in Fig. \ref{fig:score-graph-example}.}
\label{fig:graphical-model-example}%
\end{figure}

Denoting by $\bpTest_{\edges_\inter}$ the vector of all the random variables $\pTest_{\idxi\idxj},(\idxi,\idxj)\in\edges_\inter$, the joint distribution 
factorizes as
\[
\Prob(\bpTest_{\edges_\inter},\pState_1,\dots,\pState_\nNodes)\! = \!\Big(\prod_{(\idxi,\idxj)\in\edges_\inter}\!\!\!\!\Prob(\pTest_{\idxi\idxj}|\pState_\idxi,\pState_\idxj)\Big)\Big(\prod_{\idxi=1}^\nNodes\Prob(\pState_\idxi)\Big).
\]
We assume $\pTest_{\idxi\idxj}$, $(\idxi,\idxj)\in\edges_\inter$ are ruled by a
conditional probability distribution
$\Prob(\pTest_{\idxi\idxj} | \pState_\idxi,\pState_\idxj; \btensor)$, depending
on a \emph{parameter} vector $\btensor$ whose components take values in a given
set $\Tensor$. For notational purposes, we define the tensor
\begin{equation}\label{eq:p_hlm}
  p_{\idxh|\idxl,\idxm}(\btensor) := \Prob(\pTest_{\idxi\idxj} =
  \vTest_\idxh | \pState_\idxi = \vState_\idxl,\pState_\idxj = \vState_\idxm;
  \btensor),
\end{equation}
where $\vTest_\idxh\in\test$ and $\vState_\idxl, \vState_\idxm \in \state$. 
From the definition of probability distribution, we have  the constraint $\btensor\in\mathcal{S}_{\Tensor}$ with
\begin{align*}
\mathcal{S}_{\Tensor} := \Big\{\btensor\in\Tensor : p_{\idxh |
  \idxl,\idxm}(\btensor)\in[0,1], \sum_{\idxh = 1}^{\nTest} p_{\idxh
  | \idxl,\idxm}(\btensor) = 1\Big\}.
\end{align*}
To clarify the notation, an example realization of $p_{\idxh | \idxl,\idxm}(\btensor)$ for a
given $\btensor\in\mathcal{S}_{\Tensor}$, is depicted in Fig.~\ref{fig:tensor}.
\begin{figure}[!htpb]
\centering
	\includegraphics[scale=0.8]{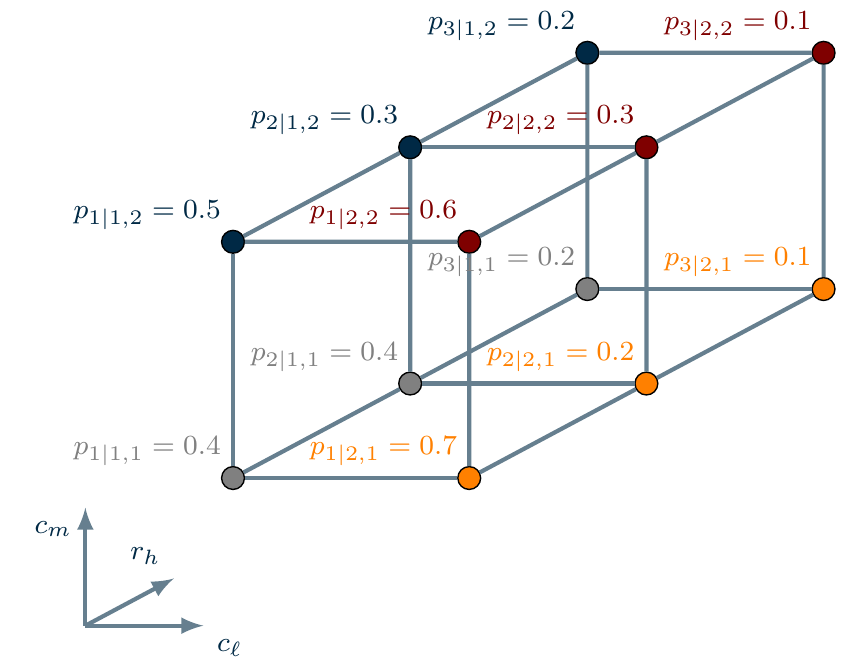}
\caption{Example of a tensor $\{p_{\idxh|\idxl,\idxm}(\btensor)\}_{\idxl,\idxm=1,2|\idxh=1,2,3}$ for fixed  $\btensor\in\mathcal{S}_{\btensor}$.
} 
\label{fig:tensor}
\end{figure}

We model $\pState_\idxi$, $\idxi=1,\dots,N$, as identically distributed random
variables ruled by a probability distribution $\Prob(\pState_\idxi ; \bro)$,
depending on a \emph{hyperparameter} vector $\bro$ whose components take
values in a given set $\Ro$. For notational purposes, we introduce
\begin{equation}\label{eq:p_l}
p_\idxl(\bro) := \Prob(\pState_\idxi = \vState_\idxl; \bro).
\end{equation}
and, analogously to $\btensor$, we have the constraint
$\bro \in \mathcal{S}_{\Ro}$ with
\[
\mathcal{S}_{\Ro} := \Big\{\bro\in\Ro : p_\idxl(\bro)
\in [0,1], \sum_{\idxl = 1}^{\nState} p_\idxl(\bro) = 1 \Big\}.
\]
We assume that $p_{\idxh | \idxl,\idxm}$ and $p_\idxl$ are continuous functions, and
that each node knows $p_{\idxh | \idxl,\idxm}, p_\idxl$ and the scores received
from its in-neighbors and given to its out-neighbors in $\graph_\inter$. 

Notice that the least structured case for the model above is
given by the \emph{categorical model} in which the vector of parameters and the
vector of hyperparameters are given by the corresponding probability
masses. That is, $\btensor$ and $\bro$ have respectively $\nTest+2\nState$ and
$\nState$ components.
We point out that the categorical model, being so unstructured, is the most
flexible one. Clearly, this flexibility is paid by a much higher number of
parameters, which quickly degenerates in over-fitting. Therefore, in practical
applications one usually exploits domain-specific knowledge to identify a
suitable parametrization in terms of the most relevant parameters and
hyperparameters. 
Some examples are discussed in the next subsection, while the
problem of jointly estimating the \emph{parameter-hyperparameter}
$(\btensor,\bro)$ will be addressed in the next section as a building block of
the (distributed) learning scheme.

\subsection{Examples of application scenarios}\label{subsec:example_scenarios}

\subsubsection{Binary-state learning for anomaly detection}\label{subsub:preparata}
We consider a network in which each node $\idxi$ tests neighboring nodes $\idxj$
with a binary outcome indicating if the tested node is deemed faulty (i.e., its
state is $\oState_\idxj = 1$) or not (i.e., $\oState_\idxj = 0$). Since each
node performing the evaluation can be itself faulty, its outcome is not always
reliable; also, no node knows whether it itself is faulty or not.
We consider a probabilistic extension of the
well-known Preparata model \cite{preparata1967connection}.
Specifically, we assume that the evaluation outcome is determined as follows: if
node $\idxi$ is working properly, then it will return the true status of the
evaluated node $\idxj$ (i.e., $\oTest_{\idxi\idxj}=1$ if node $\idxj$ is faulty
and $\oTest_{\idxi\idxj}=0$ if it is working properly); conversely, if node
$\idxi$ is faulty, the outcome is uniformly random.  Formally:
\begin{align}
\begin{split}
p_{\idxh|\idxl,\idxm} &= (1 - \vState_\idxl)\Big[(1 - \vState_\idxm)(1 -
                        \vTest_\idxh) + \vState_\idxm\vTest_\idxh\Big] +
                        \frac{1}{2}\vState_\idxl,\\
p_\idxl(\ro) &= \ro^{\vState_\idxl}(1 - \ro)^{1 -
  \vState_\idxl},\qquad\ro\in[0,1],
\end{split}
\label{eq:preparata}
\end{align}
with $\nTest = 2$ $(\vTest_1 = 0, \vTest_2 = 1)$ and $\nState = 2$ $(\vState_1 = 0, \vState_2 = 1)$.
In this first scenario, we assume for simplicity that the distribution of
evaluation results is known, so the only unknown is the hyperparameter $\ro$,
which is the \emph{a priori} probability that a node is faulty. We refer to this
model as \emph{anomaly-detection model}.

\subsubsection{Social ranking}\label{subsub:social_ranking}
Another relevant scenario is user profiling in social networks. That is, in
social relationships, people naturally tend to aggregate, tacitly or explicitly,
into groups based on some affinity. 

For example, consider an online forum on a dedicated subject, wherein each
member can express her/his preferences by assigning to posts of other
members/colleagues a score from $1$ to $\nTest$ indicating an increasing level
of appreciation for that post. 
In order to model the distribution of scores, we consider distance-based ranking
distributions, \cite{fligner1986distance}, \cite{marden1996analyzing}, in which
(ranking) probabilities decrease as far as the distance from a reference
(ranking) probability increases. 
To fit our needs, we propose for the distribution of scores the following slight variation of the
so-called Mallow's $\phi$-model (see \cite{mallows1957nonnull}):
\begin{equation}
p_{\idxh|\idxl,\idxm}(\tensor) = \frac{1}{\psi_{\idxl,\idxm}(\tensor)} e^{-\big(\frac{(\vTest_\nTest -
      \vTest_\idxh) / \vTest_\nTest - d(\vState_\idxl,\vState_\idxm) /
      \vState_\nState}{\tensor}\big)^2},
\label{eq:mallow}
\end{equation}
where $\vTest_\idxh = \idxh$ ($\idxh = 1,\dots,\nTest$), $\vState_\idxl = \idxl$
($\idxl = 1,\dots,\nState$), $\tensor\in\mathbb{R}_{>0}$ is a dispersion parameter,
$\psi_{\idxl,\idxm}(\tensor)$ is a normalizing constant, and $d$ is a
semi-distance, i.e., $d \ge 0$ and $d(\vState_\idxl,\vState_\idxm) = 0$ iff
$\vState_\idxl=\vState_\idxm$.
Informally, the ``farther'' a given community $\vState_\idxl$ is from another
community $\vState_\idxm$, the higher will be the distance
$d(\vState_\idxl,\vState_\idxm)$, and thus the lower the score.

In many cases the resulting subgroups
reflect some hierarchy in the population. Basic examples could be forums or
working teams. Thus, we consider a scenario in which each person belongs to a community
reflecting some degree of expertise about a given topic or field.  In
particular, we have $\nState$ ordered communities, with $\ell$th community given
by $\vState_\idxl = \idxl$. That is, for example, a person in the community
$\vState_1$ is an \emph{newbie}, while a person in $\vState_\nState$ is a
\emph{master}.
Since climbing in the hierarchy is typically the result of several promotion
events, a natural probabilistic model for the communities is a binomial
distribution $\mathcal{B}(\nState - 1,\ro)$, where $\ro\in[0,1]$ represents the
probability of being promoted, i.e.,
\begin{align*}
p_\idxl(\ro) &= \binom{\nState - 1}{\vState_\idxl - 1}\ro^{\vState_\idxl - 1}(1
               - \ro)^{\nState - 1 - (\vState_\idxl - 1)}.
\end{align*}
We will refer to this second set-up as  \emph{social-ranking model}.

\section{Interaction-based distributed learning}\label{sec:algorithm}
In this section we describe the proposed distributed learning scheme. 
Without loss of generality, we concentrate on a set-up in which a node wants to
self-classify. The same scheme also applies to a scenario in which a node wants
to classify its neighbors, provided it knows their given and received
scores. Notice that, in many actual contexts, as, e.g., in social network
platforms, this information is readily available.

The section is structured as follows.
First, we derive a local Bayesian classifier provided that an estimation of
parameter-hyperparameter $(\btensor,\bro)$ is available.
Then, based on a combination of plain ML and Empirical Bayes estimation
approaches, we derive a joint parameter-hyperparameter estimator. 
Finally, we propose two suitable relaxations of the Score Bayesian Network which
lead to distributed estimators, based on proper distributed
optimization algorithms.


\subsection{Bayesian classifiers (given parameter-hyperparameter)}
Each node can self-classify (i.e., learn its state) if an estimate
$(\bm{\hat{\tensor}}, \bm{\hat{\ro}})$ of
parameter-hyperparameter $(\btensor, \bro)$ is
available. Before discussing in details how this estimate can be obtained in a distributed way, we develop a \emph{decentralized} MAP self-classifier
which uses only single-hop information, i.e., the scores it gives to and receives from neighbors.

Formally, let $\bm{\oTest}_{N_\idxi}$ be the vector of (observed) scores that
agent $\idxi$ obtains by in-neighbors and provides to out-neighbors, i.e., the
stack vector of $\oTest_{\idxj\idxi}$ with $(\idxj,\idxi)\in\edges_\inter$ and
$\oTest_{\idxi\idxj}$ with $(\idxi,\idxj)\in\edges_\inter$. Consistently, let
${\bm{\pTest}\!}_{N_\idxi}$ be the corresponding random vector, then for each
agent $\idxi=1,\dots,\nNodes$, we define
\begin{equation*}
u_\idxi(\vState_\idxl) := \Prob(\pState_\idxi = \vState_\idxl |
{\bm{\pTest}\!}_{N_\idxi} = \bm{\oTest}_{N_\idxi}; \bm{\hat{\ro}},
\bm{\hat{\tensor}}), \quad \idxl = 1,\dots,\nState. 
\end{equation*}
The \emph{soft classifier} of $\idxi$ is the probability vector
$\bm{u}_\idxi := (u_\idxi(\vState_1),\dots,u_\idxi(\vState_\nState))$
(whose components are nonnegative and sum to $1$). In
Fig.~\ref{fig:soft-classifier} we depict a pie-chart representation of an example vector
$\bm{u}_\idxi$.

\begin{figure}[!htpb]
\centering
	\includegraphics[scale=0.8]{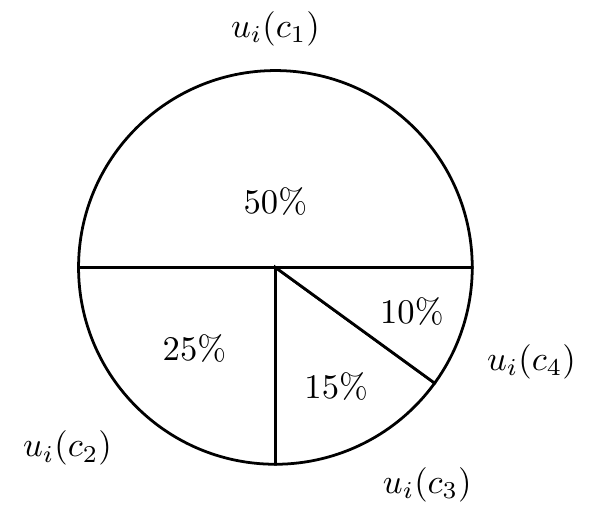}
\caption{Example of outcome of the soft classifier of an agent $i$, for $\nState = 4$: $\bm{u}_\idxi = (0.5,0.25,0.15,0.1)$.}
\label{fig:soft-classifier}
\end{figure}

From the soft classifier we can define the classical \emph{Maximum A-Posteriori
  probability (MAP) classifier} as the argument corresponding to the maximum
component of $u_\idxi$, i.e.,
\begin{equation*}
\hat{\oState}_\idxi := \argmax_{\vState_\idxl\in\state} u_\idxi(\vState_\idxl).
\end{equation*}

The main result here is to show how to efficiently compute the soft and MAP classifiers. First,
we define
\begin{align*}
\nNodes_\idxi^{\IO}&:=\{\idxj:(\idxj,\idxi)\in\edges_\inter,(\idxi,\idxj)\in\edges_\inter\},\\
\nNodes_\idxi^{\OtoI}&:=\{\idxj:(\idxj,\idxi)\in\edges_\inter,(\idxi,\idxj)\notin\edges_\inter\},\\
\nNodes_\idxi^{\ItoO}&:=\{\idxj:(\idxi,\idxj)\in\edges_\inter, (\idxj,\idxi)\notin\edges_\inter\},
\end{align*}
and for each $h,k=1,\dots,R$ we introduce the quantities:
\begin{align*}
n_\idxi^{\IO}(h,k) &:= |\{j\in\nNodes_\idxi^{\IO}: y_{ij}=r_h, y_{ji}=r_k\}|,\\
n_\idxi^{\OtoI}(h) &:= |\{j\in\nNodes_\idxi^{\OtoI}: y_{ji}=r_h\}|,\\
n_\idxi^{\ItoO}(h) &:= |\{j\in\nNodes_\idxi^{\ItoO}: y_{ij}=r_h\}|.
\end{align*}

\begin{theorem}\label{thm:MAP_characterization}
  Let $\idxi\in\nodes$ be an agent of the score graph. Then, the
  components of the vector $u_\idxi$ are given by 
\begin{equation*}
    u_{\idxi}(\vState_\idxl) = \frac{v_{\idxi}(\vState_\idxl)}{Z_{\idxi}}
\end{equation*}
where $Z_{\idxi} = \sum_{\ell = 1}^C v_{\idxi}(\vState_\idxl)$ is a normalizing
constant, and
$v_{\idxi}(\vState_\idxl) =
p_\idxl(\bm{\hat{\ro}})\pi_\idxi^{\IO}(c_\idxl)\pi_\idxi^\OtoI(c_\idxl)\pi_\idxi^\ItoO(c_\idxl)$
with
\begin{align*}
\pi_\idxi^{\IO}(c_\idxl) &= \prod_{h,k=1}^C\Big(\sum_{m=1}^C p_{k|\idxm,\idxl}(\bm{\hat{\tensor}})p_{h|\idxl,\idxm}(\bm{\hat{\tensor}})p_{\idxm}(\bm{\hat{\ro}})\Big)^{n_\idxi^{\IO}(h,k)},\\
\pi_\idxi^\OtoI(c_\idxl) &= \prod_{h = 1}^R\Big(\sum_{\idxm=1}^C p_{h|\idxm,\idxl}(\bm{\hat{\tensor}})p_\idxm(\bm{\hat{\ro}})\Big)^{n_\idxi^\OtoI(h)},\\
\pi_\idxi^\ItoO(c_\idxl) &= \prod_{h = 1}^R\Big(\sum_{\idxm=1}^C p_{h|\idxl,\idxm}(\bm{\hat{\tensor}})p_\idxm(\bm{\hat{\ro}})\Big)^{n_\idxi^\ItoO(h)}.
\end{align*}
\end{theorem}
The proof is given in Appendix~\ref{app:MAP_characterization}.

\begin{corollary}
If the score graph is undirected, then we have 
\[
v_{\idxi}(\vState_\idxl) = p_\idxl(\bm{\hat{\ro}})\pi_\idxi^{\IO}(c_\idxl).
\]
\end{corollary}

\subsection{Joint Parameter-Hyperparameter ML estimation (JPH-ML)}
Classification requires that at each node an estimate
$(\bm{\hat{\tensor}}, \bm{\hat{\ro}})$ of
parameter-hyperparameter $(\btensor, \bro)$ is available. 

On this regard, a few remarks about $\btensor$ and $\bro$ are now in
order. Depending on both the application and the network context, these
parameters may be known, or (partially) unknown to the nodes. If both of them
are known, we are in a pure Bayesian set-up in which, as just shown, each node
can independently self-classify with no need of cooperation. The case of unknown
$\btensor$ (and known $\bro$) falls into a Maximum-Likelihood framework, while
the case of unknown $\bro$ (and known $\btensor$) can be addressed by an
\emph{Empirical Bayes} approach.
In this paper we consider a general scenario in which both of them
can be unknown.
Our goal is then to compute, in a distributed way, an estimate of
\emph{parameter-hyperparameter} $(\btensor,\bro)$ and use it for the
classification at each node.

In the following we show how to compute it in a distributed way by following a mixed
Empirical Bayes and Maximum Likelihood approach.
We define the \emph{Joint Parameter-Hyperparameter Maximum Likelihood (\jphml/)
  estimator} as 
\begin{equation}
\label{eq:MLE}
(\bm{\hat{\tensor}}_{\text{\tiny ML}}, \bm{\hat{\ro}}_{\text{\tiny ML}}) :=
\argmax_{(\btensor, \bro)\in
  \mathcal{S}_{\Tensor}\times\mathcal{S}_{\Ro}}
\LF({\bm{\oTest}\!}_{\edges_\inter};\btensor, \bro)
\end{equation}
where ${\bm{\oTest}\!}_{\edges_\inter}$ is the vector of all
scores $\oTest_{\idxj\idxi}, (\idxj,\idxi)\in\edges_\inter$, and 
\begin{equation}
\label{eq:likelihood}
\LF({\bm{\oTest}\!}_{\edges_\inter};\btensor, \bro) =
\Prob({\bm{\pTest}\!}_{\edges_\inter} = {\bm{\oTest}\!}_{\edges_\inter} \, ; \, \btensor, \bro)
\end{equation}
is the \emph{likelihood function}. 

Notice that, while $\btensor$ is directly linked to the observables
${\bm{\oTest}\!}_{\edges_\inter}$, the hyperparameter $\bro$ is related to the
unobservable states. While one could readily obtain the likelihood
function for the sole estimation of $\btensor$ from the distribution of scores,
the presence of $\bro$ requires to marginalize over all unobservable state
(random) variables.
Thus, by using the law of total probability
\begin{align}
\begin{split}
&\LF({\bm{\oTest}\!}_{\edges_\inter};\btensor, \bro) = \\
&\hspace{5pt}\sum_{\idxl_1=1}^\nState\! \cdots\! \sum_{\idxl_\nNodes = 1}^\nState
      \!\Prob({\bpTest\!}_{\edges_\inter}\!=\! {\boTest\!}_{\edges_\inter},
       \pState_{1}\!=\!\vState_{\idxl_1},\dots,\pState_{\nNodes}\!=\!\vState_{\idxl_\nNodes}).
\end{split}
\label{eq:LF_deriv_1}
\end{align}
Indicating with $\nNodes^{I}_{\idxi}$ the set of in-neighbors of agent $\idxi$
in the score graph (we are assuming that it is non-empty), the probability in
\eqref{eq:LF_deriv_1} can be written as the product of
the conditional probability of scores, i.e.,
\begin{align*}
  \Prob({\bpTest\!}_{\edges_\inter}= {\boTest\!}_{\edges_\inter}\,|\,
          &\pState_{1}=\vState_{\idxl_1},\dots, \pState_{\nNodes}=\vState_{\idxl_\nNodes})
          = \\
  &\prod_{\idxi =
  1}^\nNodes\prod_{\idxj\in\nNodes^{I}_{\idxi}} \Prob(\pTest_{\idxj\idxi} =
  \oTest_{\idxj\idxi}| \pState_{\idxj} = \vState_{\idxl_\idxj},
  \pState_{\idxi} = \vState_{\idxl_\idxi})  
\end{align*}
multiplied by the prior probability of states, i.e.,
\begin{align*}
\Prob(\pState_{1}=\vState_{\idxl_1},\dots, \pState_{\nNodes}=\vState_{\idxl_\nNodes})
  = \prod_{\idxi=1}^{\nNodes}\Prob(\pState_{\idxi}=\vState_{\idxl_\idxi}).
\end{align*}
Thus, the likelihood function turns out to be
\begin{align*}
  \LF({\bm{\oTest}\!}_{\edges_\inter};\btensor, \bro)
\!=\! \sum_{\idxl_1=1}^\nState \!\cdots \!\sum_{\idxl_\nNodes =
  1}^\nState\prod_{\idxi=1}^\nNodes
p_{\idxl_\idxi}(\bro)\prod_{\idxj\in\nNodes^{I}_{\idxi}}p_{h_{ji}\,|\,
  \idxl_\idxi,\idxl_\idxj }(\btensor)
\end{align*}
where $\idxh_{\idxj\idxi}$ is the index of the score element
$\vTest_{\idxh_{\idxj\idxi}} \in \test=\{\vTest_1,\ldots,\vTest_\nTest\}$ associated to the score
$\oTest_{\idxj\idxi}$, i.e., $\oTest_{\idxj\idxi}= \vTest_{\idxh_{\idxj\idxi}}$.

\subsection{Distributed JPH Node-based Relaxed estimation (JPH-NR)}
From the equations above it is apparent that the likelihood function couples the
information at all nodes, so problem~\eqref{eq:MLE} is not amenable to
distributed solution.  To make it distributable, we propose a relaxation
approach: in particular, we seek for a minimal relaxation in terms of
dependencies of the joint probabilities that results in a separable structure
of the likelihood function.  To this aim we introduce, instead of
$\LF({\bm{\oTest}\!}_{\edges_\inter};\btensor, \bro)$, a \emph{Node-based
  Relaxed (NR) likelihood}
$\LF_{NR}({\bm{\oTest}\!}_{\edges_\inter};\btensor, \bro)$.
Let $\bm{\oTest}_{N_\idxi^I}$ be the vector of (observed) scores that agent
$\idxi$ obtains by in-neighbors and ${\bm{\pTest}\!}_{N_\idxi^I}$ the
corresponding random vector. Then,
\begin{align}
 \LF_{NR}({\bm{\oTest}\!}_{\edges_\inter};\btensor, \bro)
:= \prod_{\idxi = 1}^\nNodes \Prob({\bm{\pTest}\!}_{N_\idxi^I} = {\bm{\oTest}\!}_{N_\idxi^I} \, ; \, \btensor,\bro).
\end{align}
This relaxation can be interpreted as follows. We pretend that each node has a
virtual state, independent from its true state, every time it evaluates another
node. Thus, in the Score Bayesian Network, besides the state variables
$\pState_\idxi$, $\idxi = 1,\dots,N$, there will be additional variables
$\pState_\idxi^{\rightarrow \idxj}$ for each $j$ with $(i,j)\in\edges_\inter$.
To clarify this model,
Figs.~\ref{fig:score-graph-middle}-\ref{fig:graphical-model-middle} depict the
node-based relaxed graph and the corresponding graphical model for the same example
given in Figs.~\ref{fig:score-graph-example}-\ref{fig:graphical-model-example}.

\begin{figure}[!htpb]
	\centering
	\includegraphics{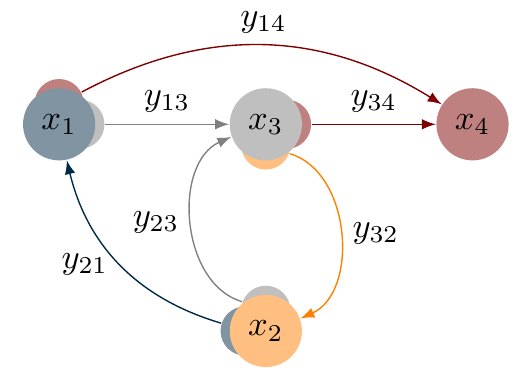}
	\caption{Node-based relaxation of the score graph in Fig. \ref{fig:score-graph-example}, with virtual nodes indicating the virtual states of each node.}
		\label{fig:score-graph-middle}
\end{figure}
\begin{figure}[!htpb]
	\centering
	\includegraphics{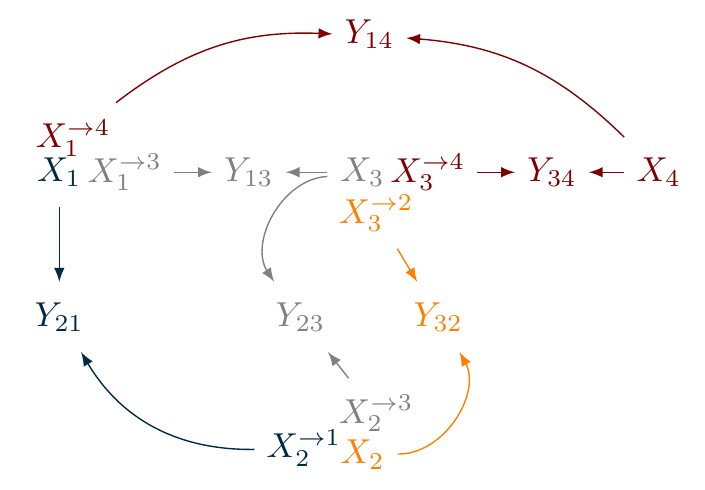}
	\caption{Node-based relaxation of the score Bayesian network of Fig. \ref{fig:score-graph-middle}.}
		\label{fig:graphical-model-middle}
\end{figure}

Since ${\bm{\pTest}\!}_{N_\idxi^I}$, $\idxi=1,\dots,N$, are not independent, then
clearly $\LF\neq\LF_{NR}$. 
However, as it will appear from numerical performance assessment, reported in
the Section~\ref{sec:simulations}, this choice yields reasonably small
estimation errors.

Using this virtual independence between ${\bm{\pTest}\!}_{N_\idxi^I}$, with
$\idxi=1,\dots,N$, we define the \emph{JPH-NR estimator} as
\begin{equation}
\label{eq:MLE-NR}
(\bm{\hat{\tensor}}_{\text{\tiny NR}}, \bm{\hat{\ro}}_{\text{\tiny NR}}) :=
\argmax_{(\btensor, \bro)\in
  \mathcal{S}_{\Tensor}\times\mathcal{S}_{\Ro}}
\LF_{NR}({\bm{\oTest}\!}_{\edges_\inter};\btensor, \bro).
\end{equation}

The next result characterizes the structure of \JPHNR/ \eqref{eq:MLE-NR}.  
\begin{proposition}
\label{prop:loglike_NR}
The \JPHNR/ estimator based on the node-based relaxation of the
score Bayesian network is given by
\begin{equation}
\label{eq:NRopt}
(\bm{\hat{\tensor}}_{\text{\tiny NR}}, \bm{\hat{\ro}}_{\text{\tiny NR}}) =
\argmax_{(\btensor, \bro)\in
  \mathcal{S}_{\Tensor}\times\mathcal{S}_{\Ro}}
\sum_{\idxi=1}^{\nNodes} g(\btensor,\bro;\bm{\nEdges}_i)
\end{equation}
with $\bm{\nEdges}_i \!=\! [\nEdges_\idxi^{(1)} \cdots
\nEdges_\idxi^{(\nTest)}]^\top$,   $\nEdges_\idxi^{(\idxh)} \!:=\! |\{\idxj\in
N^I_\idxi: y_{\idxj\idxi} = r_h\}|$, and
\begin{align}\label{eq:g}
g(\btensor,\bro;\bm{n}_i) &=\nonumber\\ 
\log\Big(\sum_{\idxl = 1}^C&  p_\idxl(\bro)\prod_{h =
  1}^R \Big(\sum_{m=1}^C p_{h \,|\, m, \idxl}(\btensor) p_m(\bro)\Big)^{n_i^{(h)}}\Big).
\end{align}
\end{proposition}
The proof is given in Appendix \ref{app:loglike_NR}.

Proposition~\ref{prop:loglike_NR} ensures that the \JPHNR/ estimator can be
computed by solving an optimization problem that has a separable cost (i.e., the
sum of $N$ local costs). Available distributed optimization algorithms for
asynchronous networks can be adopted to this aim, e.g.
\cite{carli2015analysis}, \cite{nedic2015distributed}, \cite{di2016next}.

\subsection{Distributed JPH Fully-Relaxed estimation (JPH-FR)}
Although \JPHNR/ estimator is a viable solution for which we will report simulation results
later in the paper, we consider a stronger relaxation, which gives rise to a
more convenient distributed algorithm consisting of a linear (consensus-like)
averaging process and a purely local optimization step.

Thus, we introduce the \emph{Fully Relaxed (FR) likelihood}:
\begin{align}\label{eq:L_FR}
 \LF_{FR}({\bm{\oTest}\!}_{\edges_\inter};\btensor, \bro)
:= \prod_{(\idxi,\idxj)\in\edges_\inter} \Prob(\pTest_{\idxi\idxj} = \oTest_{\idxi\idxj}\, ; \,\btensor,\bro)
\end{align}
where all dependencies among the variables
$\pTest_{\idxi\idxj},(\idxi,\idxj)\in\edges_\inter$ are neglected.
Accordingly, the \emph{JPH-FR estimator} is given by
\begin{equation}
\label{eq:MLE-FR}
(\bm{\hat{\tensor}}_{\text{\tiny FR}}, \bm{\hat{\ro}}_{\text{\tiny FR}}) :=
\argmax_{(\btensor, \bro)\in
  \mathcal{S}_{\Tensor}\times\mathcal{S}_{\Ro}}
\LF_{FR}({\bm{\oTest}\!}_{\edges_\inter};\btensor, \bro).
\end{equation}

The following proposition exposes the structure of  \eqref{eq:MLE-FR}.
\begin{proposition} 
\label{prop:loglike_FR}
The \JPHFR/ estimator based on the full relaxation of the score Bayesian network is given by
\begin{align}\label{eq:optimization_problem}
(\bm{\hat{\tensor}}_{\text{\tiny FR}}, \bm{\hat{\ro}}_{\text{\tiny FR}}) =
\argmin_{(\btensor, \bro)\in
  \mathcal{S}_{\Tensor}\times\mathcal{S}_{\Ro}} \bm{\phi}^\top\bm{g}(\btensor, \bro)
\end{align}
where 
$
\bm{g} := (g_{(1)},\dots,g_{(\nTest)}),\quad\bm{\phi} := (\phi^{(1)},\dots,\phi^{(\nTest)}),
$
and
\begin{align*}
  g_{(h)}(\btensor, \bro) &:=
                                                    -\log\Big(\sum_{\idxl,\idxm=
                                                    1}^\nState p_{\idxh|\idxl,\idxm}(\btensor)
                                                    p_{\idxl}(\bro)p_{\idxm}(\bro)\Big),\\
  \phi^{(\idxh)} &:= \sum_{\idxi = 1}^\nNodes \frac{\nEdges_\idxi^{(\idxh)}}{\nEdges}, \quad \idxh = 1,\dots,\nTest
\end{align*} 
\end{proposition}
The proof is given in Appendix \ref{app:loglike_FR}.

In order to solve the optimization problem \eqref{eq:optimization_problem} in a
distributed way, each agent in the network needs to know the vector $\bm{\phi}$.
A naive approach is to first run a consensus algorithm to obtain approximated
local ``copies'' of
$\bm{\phi}$; 
then, \eqref{eq:optimization_problem} can be solved by applying a standard
(centralized) optimization method, e.g. the projected gradient method.
However, in this approach one needs to wait for the consensus algorithm to
converge (up to the required accuracy), then to start another iterative
(local) procedure to finally obtain the solution. We propose here a different
approach, where only a single iterative (distributed) procedure is run. The idea
is to combine one step of consensus with one step of gradient in order to build
a sequence which converges to an optimal solution. 

Let $\nEdges_\idxi$ be the
number of incoming edges of agent $\idxi$ into the score graph $\graph_\inter$,
i.e., $\nEdges_\idxi = |N_\idxi^I|$. For each $\idxt\in\mathbb{Z}_{\geq 0}$,
agent $\idxi$ stores in memory two local states
$\bm{\xi}_\idxi(t) =
(\xi_\idxi^{(1)}(\idxt),\dots,\xi_\idxi^{(\nTest)}(t))$
and $\eta_\idxi(\idxt)$, an estimate $\bm{\phi}_\idxi(\idxt)$ of
$\bm{\phi}$, and an estimate
$(\bm{\hat{\tensor}}_{\idxi}(t),\bm{\hat{\ro}}_{\idxi}(t))$ of
$(\bm{\hat{\tensor}},\bm{\hat{\ro}})$.

By following the push-sum consensus algorithm to compute averages in directed
graphs \cite{benezit2010weighted}, we provide the following distributed algorithm to compute
$\bm{\phi}$. By denoting $d_\idxj(\idxt) = |\nNodes_{\comm,\idxi}^{O}(\idxt)|$ the
out-degree of node $\idxj$ at time $t$ in the communication graph
$\graph_\comm(t)$, node $i$ implements
\begin{align}
\begin{split}
\xi_i^{(\idxh)}(t+1) &= \sum_{j\in N_i^{\comm,I}(t)} \frac{\xi_j^{(\idxh)}(t)}{d_j(t)}\\
  \eta_i(t+1) &= \sum_{j\in N_\idxi^{\comm,I}(t)} \frac{\eta_j(t)}{d_j(t)}\\
  \phi_i^{(\idxh)}(t+1) &= \frac{\xi_i^{(\idxh)}(t+1)}{\eta_i(t+1)}\\
\end{split}
  \label{eq:push_sum}
\end{align}
with $\xi_\idxi^{(\idxh)}(0) = \nEdges_\idxi^{(\idxh)}$ and
$\eta_\idxi(0) = \nEdges_i$.

Then, each node can use its current estimate $\phi_i^{(\idxh)}(t)$ to implement
a gradient step on the estimated cost function
$\bm{\phi}_\idxi(\idxt)^\top\bm{g}$. 
That is, let
$(\bm{\hat{\tensor}}_{\idxi,0}, \bm{\hat{\ro}}_{\idxi,0}) \in
\mathcal{S}_{\btensor} \times \mathcal{S}_{\bro}$
be a starting point for the distributed estimation algorithm, $\alpha > 0$ a
suitable step-size, then
$(\bm{\hat{\tensor}}_\idxi(0), \bm{\hat{\ro}}_\idxi(0)) =
(\bm{\hat{\tensor}}_{\idxi,0}, \bm{\hat{\ro}}_{\idxi,0})$ and
\begin{align}
\begin{split}
(\bm{\hat{\tensor}}_\idxi(\idxt + 1), \bm{\hat{\ro}}_\idxi(\idxt + 1)) &=\\
  \Big[( \bm{\hat{\tensor}}_\idxi(\idxt), \bm{\hat{\ro}}_\idxi(\idxt)&) - \alpha
                                                                       \bm{\phi}_\idxi(\idxt)^\top\nabla\bm{g}(\bm{\hat{\tensor}}_\idxi(t),
                                                                       \bm{\hat{\ro}}_\idxi(\idxt))\Big]^{+} ,
                                                                     \end{split}
                                                                       \label{eq:optim_step}
\end{align}
with $[\cdot]^+$ the (Euclidean) projection operator onto the feasible set
$\mathcal{S}_{\Tensor}\times\mathcal{S}_{\Ro}$. 

The following technical assumption is needed to ensure that this last is well
defined.
\begin{assumption}\label{ass:feasible_set}
The given sets $\Tensor$ and $\Ro$ are both subsets of finite-dimensional real-vector spaces, and the product set $\mathcal{S}_{\Tensor}\times\mathcal{S}_{\Ro}$ is compact and convex.
\end{assumption}

The convergence properties of the distributed algorithm defined by \eqref{eq:push_sum}-\eqref{eq:optim_step} are given in the following theorem.

\begin{theorem}\label{thm:constant_stepsize}
  Let Assumptions \ref{ass:connectivity} and \ref{ass:feasible_set}
  hold. Suppose that $\bm{g}$ is differentiable, and that
  $\nabla\bm{g}$ is bounded and Lipschitz continuous, with constant
  $L>0$, over the feasible set
  $\mathcal{S}_{\Tensor}\times\mathcal{S}_{\Ro}$.  Let
  $0 < \alpha < \frac{2}{L}$.
  Then, any limit point of the sequence $\{(\bm{\hat{\tensor}}_\idxi(\idxt), \bm{\hat{\ro}}_\idxi(\idxt))\}_{\idxt\in\mathbb{N}}$ generated by \eqref{eq:push_sum}-\eqref{eq:optim_step} is a
  stationary point of the objective function
  $\bm{\phi}^\top\bm{g}$ over the feasible set
  $\mathcal{S}_{\Tensor}\times\mathcal{S}_{\Ro}$.
\end{theorem}
The proof is given in Appendix~\ref{app:constant_stepsize}.

\section{Application of the framework}\label{sec:simulations}

In this section we provide numerical results for two meaningful case
studies, using the anomaly-detection and the social-ranking models described in
Section \ref{subsec:example_scenarios}. Beforehand, we analyze a special binary/binary
case for which the \JPHFR/ estimator can be computed in closed
form.
\subsection{Distributed learning for anomaly detection: binary scores and binary
states}

For the anomaly-detection model described in
Section~\ref{subsec:example_scenarios}, $\nState=\nTest=2$ is a very special
case where a closed form solution for the \JPHFR/ optimization problem can be
found. Using Lemma \ref{lem:likelihood_explicit} we compute the fully relaxed
likelihood as
\begin{equation*}
\LF_{FR}(\ro) = \Big[\frac{1}{2}\ro + \ro(1 - \ro)\Big]^{\nEdges^{(2)}}
\Big[\frac{1}{2}\ro + (1 - \ro)^2\Big]^{\nEdges^{(1)}}
\end{equation*}
whose derivative is
\begin{equation*}
\LF_{FR}'(\ro) = \nEdges\Big[\frac{1}{2}\ro + \ro(1 - \ro)\Big]^{\nEdges^{(2)} - 1} \Big[\frac{1}{2}\ro + (1 - \ro)^2\Big]^{\nEdges^{(1)} - 1} h_0(\ro)
\end{equation*}
with
\begin{multline*}
h_0(\ro) = \phi^{(2)} \Big(\frac{3}{2} - 2\ro\Big) \Big(\frac{1}{2}\ro + (1 - \ro)^2\Big)\\
 - \phi^{(1)} \Big(\frac{3}{2} - 2\ro\Big) \Big(\frac{1}{2}\ro + \ro(1 - \ro)\Big).
\end{multline*}
The first two factors have roots that give
a zero cost function, thus optimizers are obtained by studying the sign of  $h_0(\ro)$.
Its roots  are $\ro_0 = \frac{3}{4}$ and
$\ro_{1,2} = \frac{1}{4}(3 \pm \sqrt{9-16\phi^{(2)}})$, and straightforward
computations show that for $\phi^{(2)} \leq \frac{9}{16}$ both $\ro_{1}$ and
$\ro_{2}$ are (real) local maximizers. However,
$\ro_{2} = \frac{1}{4}(3 + \sqrt{9-16\phi^{(2)}}) \geq 1$. Thus, we conclude that
\[
\hat{\ro}_{\text{\tiny FR}} =
\begin{cases}
\frac{3}{4} & \text{if} \; \phi^{(2)} \geq\frac{9}{16} \\[1.2ex]
\frac{1}{4}(3 - \sqrt{9-16\phi^{(2)}}) & \text{otherwise}.
\end{cases}
\]
With the expression of $\hat{\ro}_{\text{\tiny FR}}$ in hand, we are able to
compute the soft and the MAP classifiers according to
Theorem~\ref{thm:MAP_characterization}, hence only a consensus on $\bm{\phi}$ is needed. 

Besides this special (binary-binary) case, in general it is not possible to find
a closed form, hence the proposed distributed estimators are needed in
practice. Indeed, in the next example we show that, as soon as we remove the
hypothesis that scores are binary, the cost function becomes analytically
intractable, so that numerical optimization algorithms, as the ones we are
proposing in this paper, are unavoidable.

\subsection{Distributed learning for anomaly detection: $R$-ary scores and
  binary states} 
We consider an extension of the previous scenario by allowing scores to assume
multiple values. We basically relax the fact that a normally working node gives
the exact state of the tested node in a deterministic way. We assume that the
$R$ possible scores are given according to some probability, depending e.g. on the
reliability of the test or expressing the level of trust about the
tested node. For the sake of clarity we just consider a linear trend.
Formally, we let
\begin{align*}
\nTest &\ge 2,\quad\,\,\,\,\, \vTest_\idxh = \idxh - 1,\quad \idxh = 1,\dots,\nTest,\\
\nState &= 2,\quad\,\,\,\,\, \vState_1 = 0,\quad\,\, \vState_2 = 1,
\end{align*}
and consider the following probabilistic model:
\begin{align*}
p_{\idxh|\idxl,\idxm} &= \frac{2}{R}(1 - \vState_\idxl)\Big[(1 -
                        \vState_\idxm)\Big(1 -
                        \frac{\vTest_\idxh}{\vTest_\nTest}\Big) +
                        \vState_\idxm\frac{\vTest_\idxh}{\vTest_\nTest}\Big]
                        + \frac{\vState_\idxl}{R},\\
p_\idxl(\ro) &= \ro^{\vState_\idxl}(1 - \ro)^{1 - \vState_\idxl},\qquad \ro\in[0,1].
\end{align*}
Notice that this model, hereafter referred to as the \emph{reliability model}, boils down to the Preparata model \cite{preparata1967connection} for the case of binary
score ($R=2$, see Sec. \ref{subsub:preparata}). Clearly, Lemma
\ref{lem:likelihood_explicit} with the expression above for $p_{\idxh|\idxl,\idxm}$ does not have a
closed form solution, hence we show the solution by numerical analysis.

We have performed Monte Carlo simulations, with $1000$ trials for each point, for a score graph with
$\nNodes = 300$ agents according to the probabilistic model above, for
$\nTest = 5$ and $\ro = \frac{3}{10}$. As for the score graph, we considered
a sequence of scenarios by starting from a directed cyclic configuration and
then, progressively, adding edges up to the complete graph.

We considered the reliability model to compute both the \JPHNR/ and the \JPHFR/
estimator of the hyperparameter, which in this case is $\ro = \frac{3}{10}$;
then we have used these estimators to perform the classification.  Results of
these simulations are shown in Fig.~\ref{fig:diagnosis-100-gamma}, where the
Root Mean Square Error (RMSE) of the estimates of $\ro$ is reported for the two
proposed estimators. 
\begin{figure}[!htpb]
	\centering
	\includegraphics[scale=0.8]{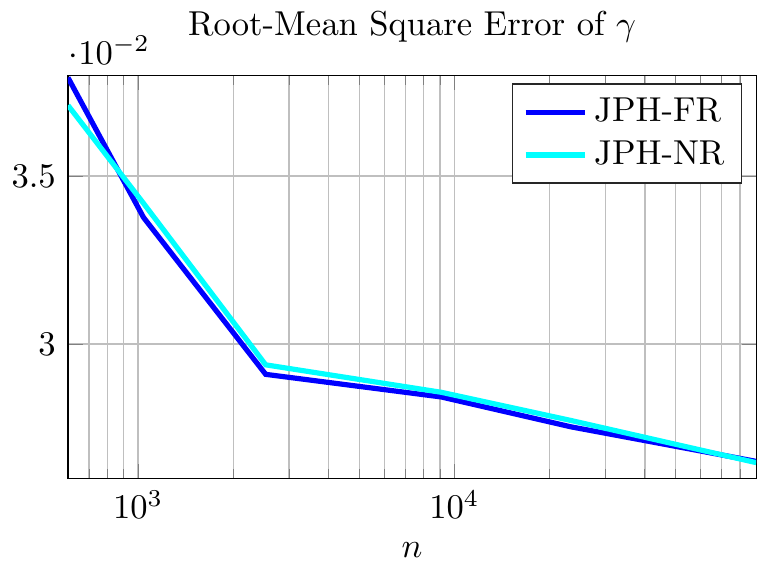}
	\caption{RMSE of the estimates of $\ro$ as a function of the number of edges $\nEdges$.}
		\label{fig:diagnosis-100-gamma}
\end{figure}
Remarkably, both estimators are able to provide a very good
estimate, which  improves as $\nEdges$ increases (as one might expect); moreover,
the \JPHFR/ is very close to the \JPHNR/.  The impact of such estimates of the
hyperparameter onto the misclassification rate is reported in
Fig.~\ref{fig:diagnosis-100}. 
\begin{figure}[!htpb]
	\centering
	\includegraphics[scale=0.8]{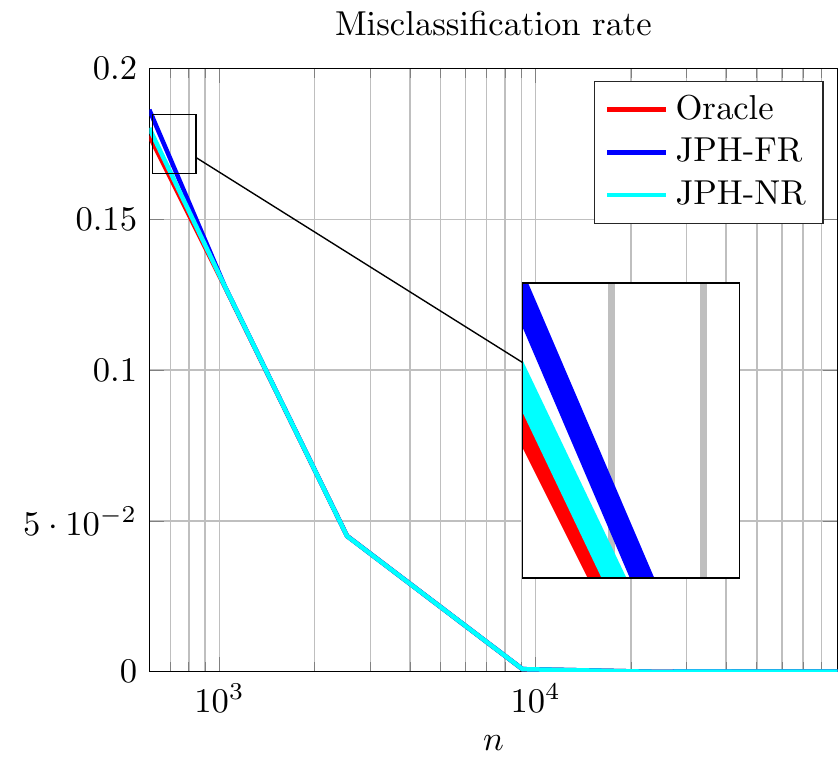}
	\caption{Misclassification rate as function of the number of edges $\nEdges$
    increasing from $\nNodes$ (cyclic graph) to $\nNodes^2-\nNodes$ (complete
    graph), with $\nNodes = 300$, $\ro = \frac{3}{10}$ and $\nTest = 5$.}
		\label{fig:diagnosis-100}
\end{figure}
As a benchmark, the curve corresponding also to
the ``oracle'' classifier that uses the true value of $\gamma$ is
reported. Results for this case are rather impressive, since both the proposed
estimator are practically attaining the performance of the ``oracle''
classifier. This indicates that the proposed relaxations are able to retain the salient aspects of the statistical relationship among the variables,  yielding results  as good as  the true graphical model.

\subsection{Distributed learning for social ranking}\label{sub:community_discovery}

In this Section we report results for the social-ranking model described in
Section~\ref{subsec:example_scenarios} with $\nState=3$, $\nTest=3$ and $\nNodes = 300$. 
We use as semi-distance in \eqref{eq:mallow}
$ d(\vState_\idxl,\vState_\idxm) = |\vState_\idxl - \vState_\idxm| = |\idxl -
\idxm|$, and as true parameter-hyperparameter $\tensor = \frac{1}{2}$ and $\ro = \frac{3}{10}$.

Monte Carlo simulations have been run to test both  \JPHNR/ and  \JPHFR/, with $1000$ trials for each point. 
Compared to the previous case, both parameter and hyperparameter are (jointly)
estimated.  Figs.~\ref{fig:ranking-100-theta}-\ref{fig:ranking-100-gamma} report
the RMSE for $\tensor$ and $\ro$, respectively, as a function of the number of
edges. It is worth noting that both the estimation errors decrease as the number
of edges increases (more data are available) and the ``less relaxed'' \JPHNR/
estimator slightly outperforms the \JPHFR/. 
\begin{figure}[!htpb]
	\centering
	\includegraphics[scale=0.8]{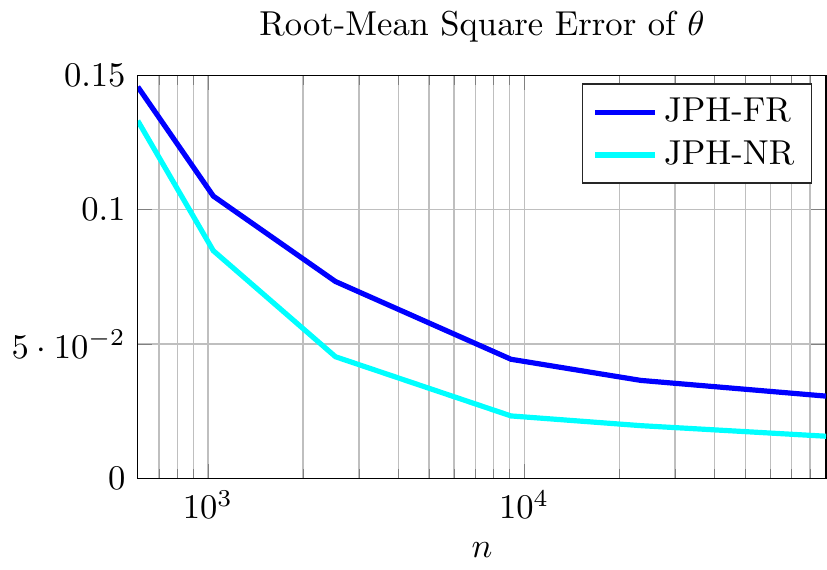}
	\caption{RMSE of the estimates of $\tensor$ as a function of the number of edges.}
	\label{fig:ranking-100-theta}
\end{figure}
\begin{figure}[!htpb]
	\centering
	\includegraphics[scale=0.8]{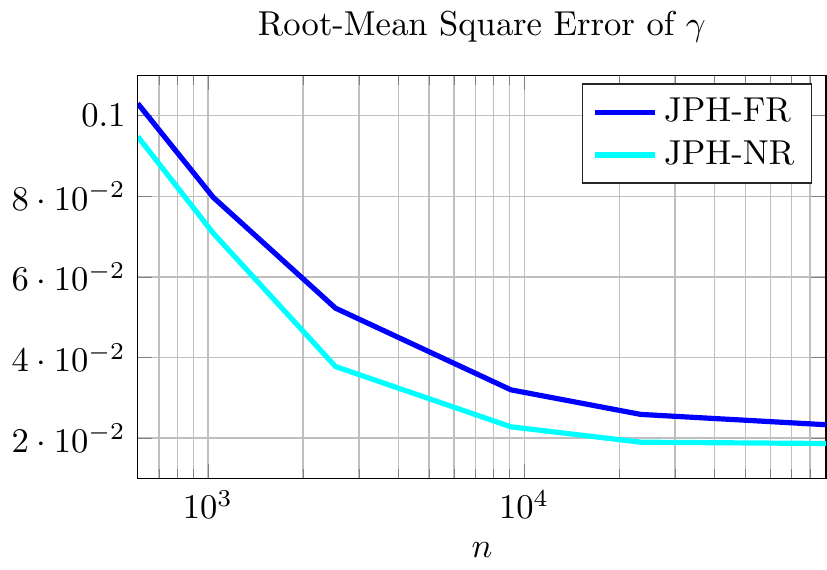}
	\caption{RMSE of the estimates of $\ro$ as a function of the number of edges.}
	\label{fig:ranking-100-gamma}
\end{figure}

The impact of estimation errors on
the learning performance is shown in Fig.~\ref{fig:ranking-100}: the curves
clearly show that, compared to the reliability model, the inferential
relationship between scores and states is ``weaker'' hence more data are needed
for a good learning. Nonetheless, the proposed estimators are still very close
to the performance of the benchmark (``oracle'').
\begin{figure}[!htpb]
	\centering
	\includegraphics[scale=0.8]{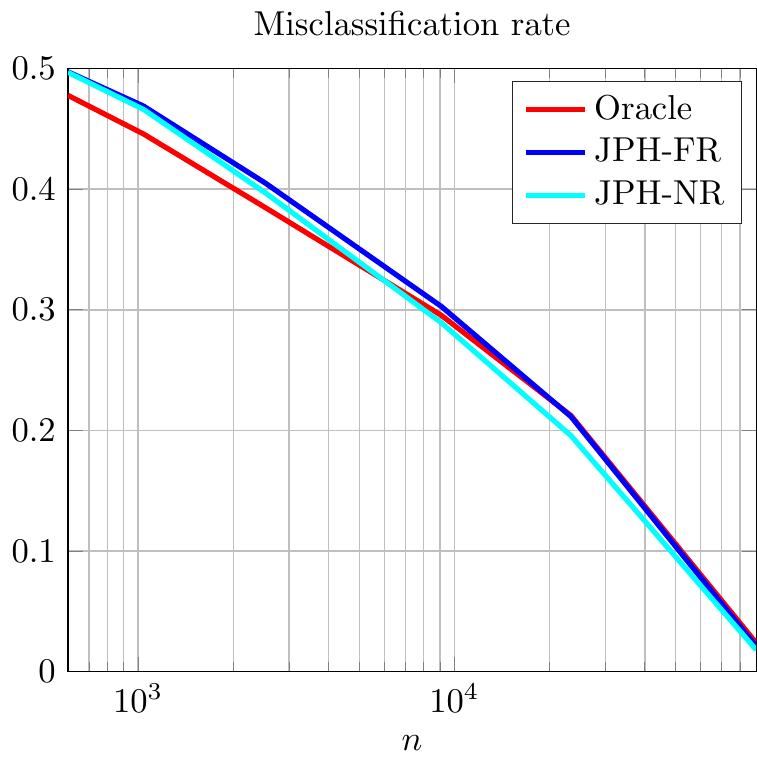}
	\caption{Misclassification rate as function of the number of edges $\nEdges$
    increasing from $\nNodes$ (cyclic graph) to $\nNodes^2-\nNodes$ (complete
    graph), with $\nNodes = 300$, $\ro = \frac{3}{10}$, $\tensor = \frac{1}{2}$, $C = 3$ and $R = 3$.}
	\label{fig:ranking-100}
\end{figure}

Finally, we report an additional case to highlight the usefulness of the soft
classifier. We considered a network of $\nNodes = 10$ agents, whose score graph
$\graph_\inter$ is shown in Fig.  \ref{fig:ranking-soft}. We drew the states and
scores in the given score graph according to the previous distributions, and
then used the social-ranking model to solve the learning problem as before, by
means of the \JPHFR/ estimator.

\begin{figure}[!htpb]
\centering
	\includegraphics[scale=0.8]{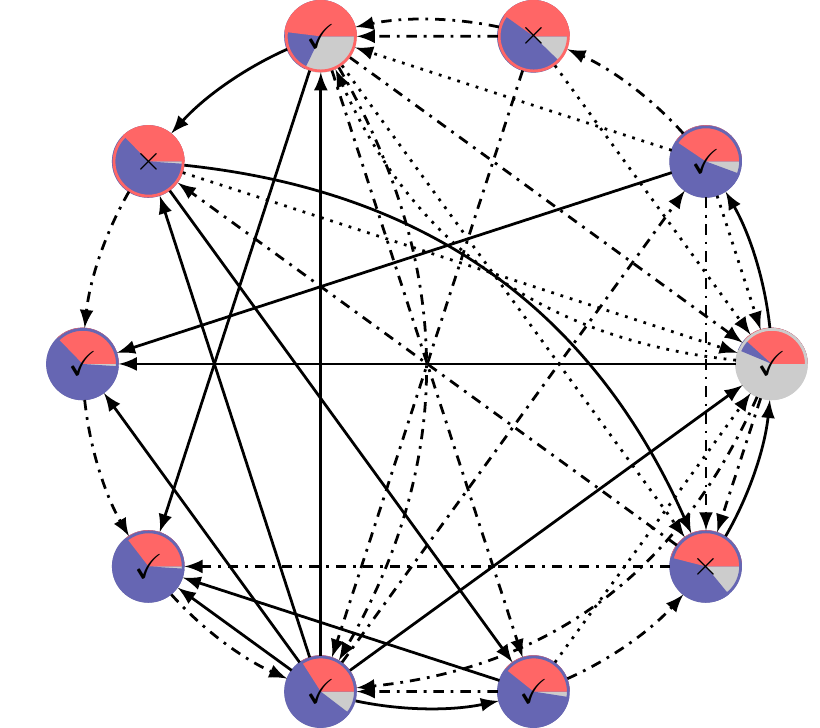}
\caption{Soft classifier representation of a particular score graph.}
\label{fig:ranking-soft}
\end{figure}

The contour of a node has a color which indicates the true state of the
node. Inside the node we have represented the outcome of the soft
classification, i.e., the output of the local self-classifier, as a
pie-chart. The colors used are: red for state $1$, blue for state $2$, gray for
state $3$. Moreover, each edge is depicted by a different pattern based on its
evaluation result $r_h$: solid lines are related to scores equal to $3$, dash
dot lines are related to scores equal to $2$, while dotted lines are related to
scores equal to $1$. We assigned to each node a symbol $\checkmark$ or $\times$
indicating if the MAP classifier correctly decided for the true state or not. We
show a realization with three misclassification errors; remarkably, all of them
correspond to a lower confidence level given by the soft classifier, which is an
important indicator of the lack of enough information to reasonably trust the
decision. It can be observed that the edge patterns concur to determine the
decision. Indeed, the only gray-state node is correctly classified thanks to the
predominant number of dotted edges insisting on it, and similarly for the
blue-state nodes which mostly have solid incoming edges. When a mix of scores
are available, clearly there is more uncertainty and the learning may fail, as
for two of the red-state nodes.

As a final remark, we point out that in some scenarios symmetries may arise in
the model, thus creating ambiguities in the labeling of the
communities. Specifically, for this scenario, it can be proven that the relaxed
likelihoods $\LF_{NR}$ and $\LF_{FR}$ take on the same value when $\gamma$ is
replaced by $1-\gamma$, e.g., $\LF_{NR}({\bm{\oTest}\!}_{\edges_\inter};\btensor, \bro) =
\LF_{NR}({\bm{\oTest}\!}_{\edges_\inter};\btensor, 1 - \bro)$. Thus the
hyperparameter estimate is not unique in this case. However, this has just the
effect of swapping the label of communities $\vState_1$ and  $\vState_3$. Notice
that, in the node-based relaxed case, agents reach directly a consensus on
the same value, thus circumventing possible inconsistencies in the labeling. For
the fully relaxed case agents can easily agree on the same value in a number of
steps equal to the diameter.

\section{Conclusion}\label{sec:conclusions}
In this paper we have proposed a novel probabilistic framework for distributed
learning, which is particularly relevant to emerging contexts such as
cyber-physical systems and social networks. In the proposed set-up, nodes of a
network want to learn their (unknown) state, as typical in classification
problems; however, differently from a classical set-up, the information does not
come from (noisy) measurements of the state but rather from observations
produced by the interaction with other nodes. For this problem we have proposed
a hierarchical (Bayesian) framework in which the parameters of the interaction
model as well as hyperparameters of the prior distributions may be unknown. Node
classification is performed by means of a local Bayesian classifier that uses
parameter-hyperparameter estimates, obtained by combining the plain ML with the
Empirical Bayes estimation approaches in a joint scheme. The resulting estimator
is very general but, unfortunately, not amenable to distributed
computation. Therefore, by relying on the conceptual tool of graphical models,
we have proposed two approximated ML estimators that exploit proper relaxations
of the conditional dependences among the involved random variables. Remarkably,
the two approximated likelihood functions do lead to distributed estimation
algorithms: specifically, for the node-based relaxed estimator, available
distributed optimization algorithms can be used, while for the fully relaxed a
faster scheme can be implemented that combines a local descent step with a
diffusion (consensus-based) step. To demonstrate the application of the proposed
schemes, we have addressed two example scenarios from anomaly detection in
cyber-physical networks and user profiling in social networks, for which Monte
Carlo simulations are reported. Results show that the proposed distributed
learning schemes, although based on relaxations of the exact likelihood
function, exhibit performance very close to the ideal classifier that has
perfect knowledge of all parameters. 

\appendix

\subsection{Preliminaries on Bayesian Networks}

Before proving Theorem~\ref{thm:MAP_characterization}, we need to recall some
useful definitions and results from graphical model theory. We want to
understand when two random variables $Z_{\text{in}}$ and $Z_{\text{fin}}$ in a
distribution associated with a Bayesian Network structure $\mathcal{K}$ are
conditionally independent given another variable $Z_g$. We first introduce a
shorthand notation. We write $Z_i\rightleftharpoons Z_{i+1}$ meaning that either
$Z_i\rightarrow Z_{i+1}$, or $Z_i\leftarrow Z_{i+1}$, or both hold.

\begin{definition}
  Given a graphical model $\mathcal{K}$, we say that $Z_0,\dots,Z_r$ form a
  \emph{trail} in $\mathcal{K}$ if, for every $i=0,\dots,r-1,$ we have
  $Z_i\rightleftharpoons Z_{i+1}$. If, for every $i=0,\dots,r-1,$ we have that
  $Z_i\rightarrow Z_{i+1}$, then the trail is called a \emph{directed path}.
\end{definition}

\begin{definition}
  We say that $Z_d$ is a \emph{descendant} of $Z$ in the graph $\mathcal{K}$ if
  there exists a directed path from $Z$ to $Z_d$.
\end{definition}

When influence can flow from $Z_{\text{in}}$ to $Z_{\text{fin}}$ via $Z_g$, we say that the
two-arrow trail $Z_{\text{in}}\rightleftharpoons Z_g\rightleftharpoons Z_{\text{fin}}$ is
\emph{active}. For each of the four possible two-arrow trails, we detail the
condition under which it is active:
\begin{itemize}
\item \emph{Causal trail} ($Z_{\text{in}}\rightarrow Z_g\rightarrow Z_{\text{fin}}$): active if and only if $Z_g$ is not observed.
\item \emph{Evidential trail} ($Z_{\text{in}}\leftarrow Z_g\leftarrow Z_{\text{fin}}$): active if and only if $Z_g$ is not observed.
\item \emph{Common cause} ($Z_{\text{in}}\leftarrow Z_g\rightarrow Z_{\text{fin}}$): active if and only if $Z_g$ is not observed.
\item \emph{Common effect} ($Z_{\text{in}}\rightarrow Z_g\leftarrow Z_{\text{fin}}$): active if and only if either $Z_g$ or one of its descendants is observed.
\end{itemize}
Now, consider the case of a longer trail
$Z_{\text{in}}=Z_0\rightleftharpoons\cdots\rightleftharpoons Z_r=Z_{\text{fin}}$. Intuitively, for influence
to \emph{flow} from $Z_{\text{in}}$ to $Z_{\text{fin}}$, it needs to flow through every single node
on the trail. In other words, $Z_{0}$ can influence $Z_r$ if for every
$i=1,\ldots, r-1$,
then $Z_{i-1}\rightleftharpoons Z_i\rightleftharpoons Z_{i+1}$ is active.

Obviously, it can happen that there is more than one trail between two nodes; in
these cases one node can
influence another if and only if there exists a trail along which influence can
flow. 
If there is no active trail between two random variables $Z_{\text{in}}$ and $Z_{\text{fin}}$, given
random variable $Z_g$, they are said to be \emph{d-separated}.

\subsection{Proof of Theorem~\ref{thm:MAP_characterization}}\label{app:MAP_characterization}
We are now ready to give the following lemma needed for the proof of Theorem~\ref{thm:MAP_characterization}.
\begin{lemma}\label{lem:independency}
Let $\idxi,\idxj,k\in\{1,\ldots,\nNodes\}$ with $\idxj\neq k$. The following statements hold:
\begin{enumerate}
\item if $(\idxi,\idxj),(k,\idxi) \in\edges_\inter$
  then $\pTest_{\idxi\idxj}$ and $\pTest_{k\idxi}$ are conditionally independent
  given $\pState_\idxi$;
\item if $(\idxi,\idxj),(\idxi,k) \in\edges_\inter$ then $\pTest_{\idxi\idxj}$
  and $ \pTest_{\idxi k}$ are conditionally independent given $\pState_\idxi$;
\item if $(\idxj,\idxi),(k,\idxi) \in\edges_\inter$ then $\pTest_{\idxj\idxi}$
  and $\pTest_{k\idxi}$ are conditionally independent given $\pState_\idxi$;
\end{enumerate}
\begin{proof} We prove only the first statement, the other two can be proven in
  the same way. Consider a trail
  $\pTest_{\idxi\idxj} = Z_0\rightleftharpoons\cdots\rightleftharpoons
  Z_r=\pTest_{k\idxi}$. Denoting by $s$ the number of state random
  variables traversed along the trail, then the length of the trail (number of
  arrows) is $r = 2s$. This property can be easily visualized in
  Figure~\ref{fig:graphical-model-example}. For each $u=0,\dots,s-1$ it
  results:
  \begin{align*}
    Z_{2u}\leftarrow Z_{2u+1}\rightarrow Z_{2u+2}
  \end{align*}
  with
  \[
  Z_{2u} = \pTest_{\idxi_u\idxj_u},\quad Z_{2u+1} = \pState_{v_u},\quad Z_{2u+2} = \pTest_{\idxi_{u+1}\idxj_{u+1}},
  \]
where we have that
$v_u\in\{\idxi_u,\idxj_u\}\cap\{\idxi_{u+1},\idxj_{u+1}\}$, while
$(\idxi_u,\idxj_u),(\idxi_{u+1},\idxj_{u+1})\in\edges_\inter$. 

Our objective is to prove that the previous trail is blocked (i.e., not active); this will imply
that $\pTest_{\idxi\idxj}, \pTest_{k\idxi}$ are $d$-separated given
$\pState_\idxi$, so that the proof follows, \cite{pearl1996identifying}.
We observe that if $\pState_{v_0} = \pState_\idxi$, then we have,
inside the trail, the common cause $Z_0\leftarrow \pState_\idxi\rightarrow Z_2$
in which $\pState_\idxi$ is observed; thus the previous common cause is blocked,
implying that also the trail is blocked.
Next, we prove that block occurs at most for $u=1$. Consider $\pState_{v_0} \neq \pState_\idxi$. In this case, from the assumption $\idxj\neq k$, by contradiction, we find
that $s > 1$. By truncating the trail at the fourth element, we have:
\begin{align*}
Z_0\leftarrow X_{v_0}\rightarrow Z_2\leftarrow \pState_{v_1}.
\end{align*}
In the common effect
$\pState_{v_0}\rightarrow Z_2\leftarrow
\pState_{v_1}$, $Z_2$ is not observed, and $Z_2 = \pTest_{\idxi_1\idxj_1}$ has no descendant in
the graphical model; thus the previous common effect is blocked, implying that
also the trail is blocked.
\end{proof}
\end{lemma}

After these preliminaries, we can proceed with the proof of Theorem~\ref{thm:MAP_characterization}. For the sake
of clarity of notation, we will omit the dependency on $\btensor$ and $\bro$.

From the Bayes theorem, we know that:
\[u_{\idxi}(\vState_\idxl) = \frac{v_{\idxi}(\vState_\idxl)}{\Prob(\bm{\pTest}_{N_\idxi} = \bm{\oTest}_{N_\idxi})},
\]
where
$v_{\idxi}(\vState_\idxl) := \Prob(\bm{\pTest}_{N_\idxi} =
\bm{\oTest}_{N_\idxi},\pState_\idxi = \vState_\idxl)$.
Our goal is to prove that
$v_{\idxi}(\vState_\idxl) =
p_\idxl(\bm{\hat{\ro}})\pi_\idxi^{\IO}(c_\idxl)\pi_\idxi^\OtoI(c_\idxl)\pi_\idxi^\ItoO(c_\idxl)$.
First of all, from the chain rule we have:
\begin{align}\label{eq:originalMAP}
v_{\idxi}(\vState_\idxl) &= \Prob(\pState_\idxi = c_\idxl) \Prob(\bm{\pTest}_{N_\idxi} = \bm{\oTest}_{N_\idxi}|\pState_\idxi = \vState_\idxl),
\end{align}
and from Lemma \ref{lem:independency}, we know that
\begin{align}\label{eq:MAPfactorization}
\Prob(\bm{\pTest}_{N_\idxi} \!=\!
\bm{\oTest}_{N_\idxi}|\pState_\idxi \!=\! \vState_\idxl) &=
\Prob(\bm{\pTest}_{N_\idxi^{\IO}} \!=\! \bm{\oTest}_{N_\idxi^{\IO}}|\pState_\idxi \!=\! c_\idxl)\nonumber\\
&\quad\times\Prob(\bm{\pTest}_{N_\idxi^{\OtoI}} \!=\!
  \bm{\oTest}_{N_\idxi^{\OtoI}}|\pState_\idxi \!=\! c_\idxl)\\
&\quad\times\Prob(\bm{\pTest}_{N_\idxi^{\ItoO}}
                                                \!=\!
                                                \bm{\oTest}_{N_\idxi^{\ItoO}}|\pState_\idxi
                                                \!=\! c_\idxl).\nonumber 
\end{align}
The next step is to study each one of the three factors. Starting from
$\Prob(\bm{\pTest}_{N_\idxi^{\IO}} = \bm{\oTest}_{N_\idxi^{\IO}}|\pState_\idxi =
c_\idxl)$, we obtain:
\begin{align}\label{eq:MAPfirstFactor}
& \!\!\!\!\!\!\!\Prob(\bm{\pTest}_{N_\idxi^{\IO}} \!\!=\!\!
  \bm{\oTest}_{N_\idxi^{\IO}}|\pState_\idxi \!\!=\!\! c_\idxl)\nonumber\\ 
= &\prod_{j\in N_i^{\IO}}\!\!\Prob(\pTest_{\idxi\idxj} \!\!=\!\! y_{ij},\pTest_{\idxj\idxi} \!\!=\!\! y_{ji}|\pState_\idxi \!\!=\!\! c_\idxl)\nonumber\\
= &\prod_{h,k=1}^R\!\! \Prob(Y_{ij_0} = r_h, Y_{j_0i} = r_k|\pState_\idxi = c_\idxl)^{n^\IO_i(h,k)}\nonumber\\
= &\prod_{h,k=1}^R\!\!\! \Big(\sum_{m = 1}^C \Prob(Y_{ij_0} \!\!=\!\! r_h, Y_{j_0i} \!\!=\!\! r_k, X_{j_0} \!\!=\!\! c_m|\pState_\idxi \!\!=\!\! c_\idxl)\Big)^{n^\IO_i(h,k)}\nonumber\\
= &\prod_{h,k=1}^R\!\! \Big(\sum_{m=1}^C \Prob(X_{j_0} \!\!=\!\! c_m)\Prob(\pTest_{\idxi\idxj_0} \!\!=\!\! r_h|\pState_{\idxi} \!\!=\!\! c_\idxl,\pState_{\idxj_0} \!\!=\!\! c_m)\nonumber\\
&\hspace{45pt}\times\Prob(\pTest_{\idxj_0\idxi} \!\!=\!\! r_k|\pState_{\idxj_0} \!\!=\!\! c_m,\pState_{\idxi} \!\!=\!\! c_\idxl)\Big)^{n^\IO_i(h,k)}.
\end{align}
In the first equation we have used again Lemma \ref{lem:independency}; in the
second equation we have used the fact that
$\pTest_{\idxi\idxj},(\idxi,\idxj)\in\edges_\inter$ are identically distributed,
and we have also aggregated the agents in $N_\idxi^{\IO}$ that receive/give
score $r_h/r_k$ from/to agent $\idxi$; in the third equation we have
marginalized with respect to the random variable $\pState_{\idxj_0}$; in the
fourth equation we have factorized according to the score Bayesian network.

For the second factor $\Prob(\bm{\pTest}_{N_\idxi^{\OtoI}} = \bm{\oTest}_{N_\idxi^{\OtoI}}|\pState_\idxi = c_\idxl)$, we obtain:
\begin{align}\label{eq:MAPsecondFactor}
~\Prob(\bm{\pTest}&_{N_\idxi^{\OtoI}} =
  \bm{\oTest}_{N_\idxi^{\OtoI}}|\pState_\idxi = c_\idxl) \nonumber\\
= &\prod_{j\in N_\idxi^{\OtoI}} \Prob(\pTest_{\idxj\idxi} = y_{ji}|\pState_\idxi = c_\idxl)\nonumber\\
= &\prod_{h=1}^R \Prob(Y_{j_0i} = r_h|\pState_\idxi = c_\idxl)^{n_i^{\OtoI}(h)}\nonumber\\
= &\prod_{h=1}^R \Big(\sum_{m=1}^C \Prob(Y_{j_0i} = r_h, X_{j_0} = c_m|\pState_\idxi = c_\idxl)\Big)^{n_i^{\OtoI}(h)}\nonumber\\
= &\prod_{h=1}^R \Big(\sum_{m=1}^C \Prob(\pState_\idxi = c_m)\nonumber\\
&\hspace{25pt}\times\Prob(\pTest_{\idxj_0\idxi} = r_h|\pState_{\idxj_0} = c_m,\pState_\idxi = c_\idxl)\Big)^{n_i^{\OtoI}(h)}.
\end{align}
Here, we point out that, when using Lemma \ref{lem:independency} in the first
equation, all the random variables $\bm{\pTest}_{N_\idxi^{\OtoI}}$ are
conditionally independent given $\pState_\idxi$. Also, factors in the second
equation are aggregated based on the agents in the set $N_\idxi^{\OtoI}$ that
give score $r_h$ to agent $\idxi$.

Finally, the third factor turns out to be:
\begin{align}\label{eq:MAPthirdFactor}
~\Prob(\bm{\pTest}&_{N_\idxi^{\ItoO}} = \bm{\oTest}_{N_\idxi^{\ItoO}}|\pState_\idxi = c_\idxl) \nonumber\\
= &\prod_{h=1}^R \Big(\sum_{c_m=1}^C \Prob(\pState_\idxi = c_m)\nonumber\\
&\hspace{25pt}\times\Prob(\pTest_{\idxi\idxj_0} = r_h|\pState_\idxi = c_\idxl, \pState_{\idxj_0} = c_m)\Big)^{n_i^{\ItoO}(h)}.
\end{align}

Plugging together \eqref{eq:MAPfirstFactor}, \eqref{eq:MAPsecondFactor},
\eqref{eq:MAPthirdFactor} into \eqref{eq:MAPfactorization}, and then into
\eqref{eq:originalMAP}, the proof is complete.

\subsection{Proof of Proposition~\ref{prop:loglike_NR}}\label{app:loglike_NR}

We omit the dependency on $\btensor$ and $\bro$ for notational purposes.

Consider a factor $\Prob({\bm{\pTest}\!}_{N_\idxi^I} = {\bm{\oTest}\!}_{N_\idxi^I})$ of the node-based relaxed likelihood, with $\idxi$ an agent in the score graph. Marginalizing with respect to $\pState_\idxi$, we have
\begin{align}
\Prob({\bm{\pTest}\!}_{N_\idxi^I} = {\bm{\oTest}\!}_{N_\idxi^I}) = \sum_{\idxl =
  1}^C \Prob({\bm{\pTest}\!}_{N_\idxi^I} = {\bm{\oTest}\!}_{N_\idxi^I},
  \pState_\idxi = \vState_\idxl).\nonumber 
\end{align}
Now, applying the chain rule we obtain
\begin{align}
\Prob({\bm{\pTest}\!}_{N_\idxi^I} = {\bm{\oTest}\!}_{N_\idxi^I}) = \sum_{\idxl =
  1}^C \Prob(\pState_\idxi = \vState_\idxl)\Prob({\bm{\pTest}\!}_{N_\idxi^I} =
  {\bm{\oTest}\!}_{N_\idxi^I}|\pState_\idxi = \vState_\idxl). \nonumber
\end{align}
We are then ready to use Lemma \ref{lem:independency}, which implies that
\begin{align}
  \Prob({\bm{\pTest}\!}_{N_\idxi^I} = {\bm{\oTest}\!}_{N_\idxi^I}) = \sum_{\idxl
  = 1}^C \Prob(\pState_\idxi = \vState_\idxl)\!\prod_{\idxj\in
  N_\idxi^I}\!\Prob(Y_{\idxj\idxi} = y_{\idxj\idxi} | \pState_\idxi =
  \vState_\idxl).\nonumber
\end{align}
Recalling that $\pTest_{\idxi\idxj}, (\idxi,\idxj)\in\edges_\inter$ are identically distributed, we aggregate all the agents in $N_\idxi^I$ that give score $r_h$ to agent $\idxi$:
\begin{align}
\Prob({\bm{\pTest}\!}_{N_\idxi^I} \!=\! {\bm{\oTest}\!}_{N_\idxi^I}) \!=\! \sum_{\idxl = 1}^C \Prob(\pState_{\idxi} \!=\! \vState_\idxl)\prod_{h = 1}^R \Prob(Y_{\idxj_0\idxi} \!=\! r_h | \pState_\idxi \!=\! \vState_\idxl)^{n_i^{(h)}},\nonumber
\end{align}
and marginalize with respect to $\pState_{\idxj_0}$, thus obtaining
\begin{align}
&\Prob({\bm{\pTest}\!}_{N_\idxi^I} = {\bm{\oTest}\!}_{N_\idxi^I}) = \sum_{\idxl = 1}^C\! \Prob(X_i \!=\! \vState_\idxl)\nonumber\\
&\hspace{50pt}\times\prod_{h = 1}^R\! \Big(\!\sum_{m=1}^C\!\Prob(Y_{\idxj_0\idxi} \!=\! r_h, X_{j_0} \!=\! c_m | \pState_\idxi \!=\! \vState_\idxl)\Big)^{n_i^{(h)}}.\nonumber
\end{align}
From the structure of the score Bayesian network the following factorization is
obtained:
\begin{align}
\Prob({\bm{\pTest}\!}_{N_\idxi^I} \!=\! {\bm{\oTest}\!}_{N_\idxi^I}) \!&=\!
  \sum_{\idxl = 1}^C \Prob(\pState_{\idxi} \!=\! \vState_\idxl)\!\prod_{h = 1}^R\!
  \Big(\sum_{m=1}^C\Prob(\pState_{\idxj_0} \!=\! c_m)\nonumber\\
	&\hspace{10pt}\times\Prob(\pTest_{\idxj_0\idxi} \!=\! r_h | \pState_{\idxj_0} \!=\! c_m, \pState_{\idxi} \!=\! c_\idxl)\Big)^{n_i^{(h)}}. \nonumber
\end{align}

Then, the node-based relaxed likelihood $\LF_{NR}$ is the product over
$i=1,\ldots,N$ of the factors above, so that 
\begin{align}
\LF_{NR}(\btensor,\bro) = \prod_{\idxi = 1}^N \sum_{\idxl = 1}^C
  p_\idxl(\bro)\prod_{h = 1}^R \Big(\sum_{m=1}^C
  p_{\idxh|\idxm,\idxl}(\btensor)p_\idxm(\bro)\Big)^{n_i^{(h)}},\nonumber 
\end{align}
where we have used the shorthand notation introduced in \eqref{eq:p_l} and \eqref{eq:p_hlm}.

Applying the logarithm to the expression of $\LF_{NR}(\btensor,\bro)$ derived
above, it follows immediately that
$ \log(\LF_{NR}(\btensor,\bro)) = \sum_{\idxi=1}^{\nNodes}
g(\btensor,\bro;\bm{\nEdges}_i), $
with each $g(\btensor,\bro;\bm{\nEdges}_i)$ as in \eqref{eq:g}.  Finally,
since the logarithm is increasing, the maximum argument is invariant under this
transformation, thus concluding the proof.

\subsection{Proof of Proposition~\ref{prop:loglike_FR}}\label{app:loglike_FR}
Before proving Proposition~\ref{prop:loglike_FR} we give a useful lemma:
\begin{lemma}\label{lem:likelihood_explicit}
  The fully relaxed likelihood can
  be written as
\begin{equation*}
\LF_{FR}(\btensor, \bro) = \prod_{\idxh = 1}^\nTest
\Big(\sum_{\idxl,\idxm = 1}^\nState p_{\idxh|\idxl,\idxm}(\btensor)
p_{\idxl}(\bro)p_{\idxm}(\bro)\Big)^{\nEdges^{(\idxh)}},
\end{equation*}
where $\nEdges^{(\idxh)} := |\{(\idxi,\idxj)\in\edges_\inter : \oTest_{\idxi\idxj} = \vTest_\idxh\}|$.
\end{lemma}
\begin{proof} 
In what follows, we omit the dependency on $\btensor$ and $\bro$. 
We start focusing our attention on the product
$\prod_{(\idxi,\idxj)\in\edges_\inter}\Prob(\pTest_{\idxi\idxj} =
\oTest_{\idxi\idxj})$,
which gives the fully relaxed likelihood. Recalling that
$\pTest_{\idxi\idxj}, (\idxi,\idxj)\in\edges_\inter$, are identically distributed,
and aggregating the edges $(\idxi,\idxj)$ for which $\oTest_{\idxi\idxj} = r_h$,
we obtain
\begin{align*}
\prod_{(\idxi,\idxj)\in\edges_\inter}\Prob(\pTest_{\idxi\idxj} =
  \oTest_{\idxi\idxj}) = \prod_{\idxh=1}^\nTest \Prob(\pTest_{\idxi_0\idxj_0} =
  \vTest_\idxh)^{n^{(\idxh)}}, 
\end{align*}
with $(\idxi_0,\idxj_0)$ being an arbitrary edge in $\edges_\inter$.
Then, marginalizing with respect to $\pState_{\idxi_0}$ and $\pState_{\idxj_0}$,
it follows that
\begin{align*}
\prod_{(\idxi,\idxj)\in\edges_\inter}&\Prob(\pTest_{\idxi\idxj} = \oTest_{\idxi\idxj})\\ 
=& \prod_{\idxh=1}^\nTest \Big(\sum_{\idxl,\idxm=1}^\nState \Prob(\pTest_{\idxi_0\idxj_0} \!=\! \vTest_\idxh, \pState_{\idxi_0} \!=\! \vState_\idxl, \pState_{\idxj_0} \!=\! \vState_{\idxm})\Big)^{n^{(\idxh)}}.
\end{align*}
Finally, exploiting again the structure of the Bayesian network, we have that
\begin{align*}
\prod_{(\idxi,\idxj)\in\edges_\inter}\!\!\!\Prob(\pTest_{\idxi\idxj} \!=\! \oTest_{\idxi\idxj}) &\!=\!
\prod_{\idxh=1}^\nTest \!\Big(\!\sum_{\idxl,\idxm=1}^\nState \Prob(\pState_{\idxi_0} \!=\! \vState_\idxl)\Prob(\pState_{\idxj_0} \!=\! \vState_\idxm)\nonumber\\
&\hspace{5pt}\times\Prob(\pTest_{\idxi_0\idxj_0} \!=\! \vTest_\idxh| \pState_{\idxi_0} \!=\! \vState_\idxl, \pState_{\idxj_0} \!=\! \vState_{\idxm})\Big)^{n^{(\idxh)}},
\end{align*}
so that the proof follows.
\end{proof}

Now we are ready to prove Proposition~\ref{prop:loglike_FR}. Since the logarithm
is a monotone transformation, it follows straight that
\begin{align}\label{eq:ML_log_pb_1}
(\bm{\hat{\tensor}}_{\text{\tiny FR}}, \bm{\hat{\ro}}_{\text{\tiny FR}}) &= \argmin_{(\bro,\btensor)\in \mathcal{S}_{\bro}\times\mathcal{S}_{\btensor}} -\frac{1}{\nEdges}\log(\LF_{FR}(\btensor, \bro)).
\end{align}
Using Lemma \ref{lem:likelihood_explicit}, we have that:
\begin{align}\label{eq:ML_log_pb_2}
-\frac{1}{\nEdges}\log(\LF_{FR}(\btensor, \bro)) &= \nonumber\\
- \sum_{\idxh = 1}^\nTest \frac{\nEdges^{(\idxh)}}{\nEdges}& \log\Big(\sum_{\idxl,\idxm = 1}^\nState p_{\idxh|\idxl,\idxm}(\btensor) p_{\idxl}(\bro)p_{\idxm}(\bro)\Big),
\end{align}
so that it is easy to show that
\begin{equation}\label{eq:ML_log_pb_3}
n^{(\idxh)} = \sum_{\idxi = 1}^\nNodes \nEdges_\idxi^{(\idxh)}.
\end{equation}
Plugging together \eqref{eq:ML_log_pb_1}, \eqref{eq:ML_log_pb_2}, and \eqref{eq:ML_log_pb_3}, the proof follows.

\subsection{Proof of Theorem~\ref{thm:constant_stepsize}}\label{app:constant_stepsize}
Before proving Theorem~\ref{thm:constant_stepsize}, we give two  Lemmas.

\begin{lemma}\label{lem:push_sum}
Let Assumption \ref{ass:connectivity} holds, and consider $\idxi\in\nodes$. Then
there exist $C > 0$, $d\in[0,1)$ and $T>0$, such that,
\begin{equation}\label{eq:phi_bound}
\|\bm{\phi} - \bm{\phi}_\idxi(\idxt)\| \leq d^\idxt C,
\end{equation}
for all $t\geq T$.
\end{lemma}
\begin{proof}
The proof follows the same steps as in
\cite{nedic2015distributed},  details are omitted for the sake of
conciseness.
\end{proof}

\begin{lemma}\label{lem:sequences}
Let $A_\idxt, B_\idxt, C_\idxt$ be three sequences such that $B_\idxt$ is nonnegative for all $\idxt$. Assume that
\[
A_{\idxt + 1} \leq A_\idxt - B_\idxt + C_\idxt,
\]
and that the series $\sum_{\idxt=0}^\infty C_\idxt$ converges. Then either $A_\idxt\to-\infty$ or else $A_\idxt$ converges to a finite value and $\sum_{\idxt = 0}^\infty B_\idxt < \infty$.
\end{lemma}
\begin{proof}
For the proof see Lemma 1 in \cite{bertsekas2000gradient}.
\end{proof}

We now prove Theorem~\ref{thm:constant_stepsize}.
For notational purposes, we define:
\begin{align}
z_\idxi^{\idxt} &:= (\bm{\hat{\tensor}}_\idxi(\idxt), \bm{\hat{\ro}}_\idxi(\idxt)),\nonumber\\
\quad\overline{z}_\idxi^{\idxt} &:= [z_\idxi^\idxt - \alpha\bm{\phi}_i(\idxt)^\top\nabla\bm{g}(z_\idxi^\idxt)]^+,\label{eq:constant_stepsize_n3}\\
g_\phi(\btensor, \bro) &:=
\bm{\phi}^\top\bm{g}(\btensor, \bro).\nonumber
\end{align}

Let $\hat{z}_\idxi$ be a limit-point of $\{z_\idxi^\idxt\}_{\idxt\ge0}$; taking
into account that the stepsize is constant, our strategy is to prove that
$\hat{z}_\idxi$ is a stationary point of $g_\phi$ over
$\mathcal{S}_{\btensor}\times\mathcal{S}_{\bro}$ by showing that $\hat{z}_\idxi$
is a fixed point of the projected gradient method related to the objective
function $g_\phi$ and to the feasible set
$\mathcal{S}_{\btensor}\times\mathcal{S}_{\bro}$, i.e.:
\begin{equation}\label{eq:goal}
\hat{z}_\idxi = [\hat{z}_\idxi - \alpha\nabla g_\phi(\hat{z}_\idxi)]^+.
\end{equation}

Clearly $z_\idxi^{\idxt + 1} = \overline{z}_\idxi^\idxt$, thus we have:
\begin{equation}\label{eq:constant_stepsize_2}
g_\phi(z_\idxi^{\idxt + 1}) - g_\phi(z_\idxi^\idxt) = g_\phi(\overline{z}_\idxi^\idxt) - g_\phi(z_\idxi^\idxt),
\end{equation}
moreover, using the fact that $\nabla g_\phi$ is Lipschitz continuous with constant $L > 0$, it follows that:
\begin{equation}\label{eq:constant_stepsize_3}
g_\phi(\overline{z}_\idxi^\idxt) - g_\phi(z_\idxi^\idxt) \le \nabla g_\phi(z_\idxi^\idxt)^\top(\overline{z}_\idxi^\idxt - z_\idxi^\idxt) + \frac{L}{2}\|\overline{z}_\idxi^\idxt - z_\idxi^\idxt\|^2.
\end{equation}
Simple calculations show that
\begin{equation}\label{eq:constant_stepsize_4}
\nabla g_\phi(z_\idxi^\idxt) = \bm{\phi}_\idxi(\idxt)^\top\nabla\bm{g}(z_\idxi^\idxt) + (\bm{\phi} - \bm{\phi}_\idxi(\idxt))^\top\nabla\bm{g}(z_\idxi^\idxt) ,
\end{equation}
where $\nabla\bm{g}(z_\idxi^\idxt)$ is the Jacobian of $\bm{g}$ at
$z_\idxi^\idxt$. 
Plugging together equations \eqref{eq:constant_stepsize_2}, \eqref{eq:constant_stepsize_3}, and \eqref{eq:constant_stepsize_4} we obtain:
\begin{align}\label{eq:constant_stepsize_5}
g_\phi(z_\idxi^{\idxt + 1}) - &g_\phi(z_\idxi^\idxt) \le \bm{\phi}_\idxi(\idxt)^\top\nabla\bm{g}(z_\idxi^\idxt)(\overline{z}_\idxi^\idxt - z_\idxi^\idxt)\nonumber\\
+ \frac{L}{2}\|\overline{z}_\idxi^\idxt& - z_\idxi^\idxt\|^2 + (\bm{\phi} - \bm{\phi}_\idxi(t))^\top \nabla\bm{g}(z_\idxi^\idxt)(\overline{z}_\idxi^\idxt - z_\idxi^\idxt).
\end{align}
We know that $\overline{z}_\idxi^\idxt$ is the projection of $z_\idxi^\idxt - \alpha\bm{\phi}_\idxi(\idxt)^\top\nabla\bm{g}(z_\idxi^\idxt)$ onto the set $\mathcal{S}_{\bro}\times\mathcal{S}_{\btensor}$, therefore:
\[
(z_\idxi^\idxt - \alpha\bm{\phi}_\idxi(\idxt)^\top\nabla\bm{g}(z_\idxi^\idxt)
- \overline{z}_\idxi^\idxt)^\top(z - \overline{z}_\idxi^\idxt) \le 0,\qquad
z\in\mathcal{S}_{\bro}\times\mathcal{S}_{\btensor}.
\]
In particular, for $z = z_\idxi^\idxt$, remembering that $ \alpha>0$, we have:
\begin{equation}\label{eq:constant_stepsize_6}
\bm{\phi}_\idxi(\idxt)^\top\nabla\bm{g}(z_\idxi^\idxt)(\overline{z}_\idxi^\idxt - z_\idxi^\idxt) \leq -\frac{1}{\alpha}\|\overline{z}_\idxi^\idxt - z_\idxi^\idxt\|^2.
\end{equation}
From the boundedness of $\nabla\bm{g}$ and from equation \eqref{eq:phi_bound},
$\nabla g_\phi$ is bounded. Moreover, from Lemma \ref{lem:push_sum},
$\bm{\phi}_\idxi(\idxt)$ converges exponentially fast to $\bm{\phi}$. Thus,
since by assumption the feasible set is bounded, there exist $C > 0$ and
$d\in[0,1)$ such that
\begin{equation}
(\bm{\phi} - \bm{\phi}_\idxi(t))^\top
\nabla\bm{g}(z_\idxi^\idxt)(\overline{z}_\idxi^\idxt - z_\idxi^\idxt) \leq
d^\idxt C.
\label{eq:bound_grad_prod}
\end{equation}
Now, let $G$ be the following auxiliary function:
\[
G(t) := g_\phi(z_\idxi^{\idxt + 1}) - g_\phi(z_\idxi^\idxt) - d^\idxt C,
\]
combining equation \eqref{eq:bound_grad_prod} with
\eqref{eq:constant_stepsize_5} and \eqref{eq:constant_stepsize_6}, and recalling
that $0 < \alpha < \frac{2}{L}$, we have:
\begin{equation}\label{eq:constant_stepsize_8}
G(t) \le \Big(\frac{L}{2} - \frac{1}{\alpha}\Big)\|\overline{z}_\idxi^\idxt - z_\idxi^\idxt\|^2 \le 0,
\end{equation}
In particular, we have:
\[
g_\phi(z_\idxi^{\idxt + 1}) \le g_\phi(z_\idxi^\idxt) + d^\idxt C,
\]
so that, applying Lemma \ref{lem:sequences} with the sequences $A_\idxt =
g_\phi(z_\idxi^\idxt)$, $B_\idxt = 0$ and $C_\idxt = d^\idxt C$, and recalling that
the feasible set is bounded, we have that $g_\phi(z_\idxi^\idxt)$ converges to a finite value. Therefore:
\begin{equation}\label{eq:constant_stepsize_n1}
\lim_{\idxt\to\infty} G(\idxt) = 0.
\end{equation}
We have supposed that $\hat{z}_\idxi$ is a limit point of $\{z_\idxi^\idxt\}_{\idxt\ge0}$, thus:
\begin{equation}\label{eq:constant_stepsize_9}
\lim_{k\to\infty}z_\idxi^{\idxt_k} = \hat{z}_\idxi.
\end{equation}
Combining equations \eqref{eq:constant_stepsize_8}, \eqref{eq:constant_stepsize_n1} and \eqref{eq:constant_stepsize_9}, it follows that 
\begin{equation}\label{eq:constant_stepsize_n2}
\lim_{k\to\infty} \overline{z}_\idxi^{\idxt_k} = \hat{z}_\idxi.
\end{equation}
By hypothesis $\nabla \bm{g}$ is continuous; moreover, $[\cdot]^+$ is non-expansive, thus continuous. Using these facts, from Lemma \ref{lem:push_sum} and equations \eqref{eq:constant_stepsize_n3}, \eqref{eq:constant_stepsize_9}, we can conclude:
\begin{equation*}
\lim_{k\to\infty}\overline{z}_\idxi^{\idxt_k} = [\hat{z}_\idxi - \alpha\nabla g_\phi(\hat{z}_\idxi)]^+.
\end{equation*}
that plugged into \eqref{eq:constant_stepsize_n2} gives \eqref{eq:goal}, thus
concluding the proof.

\end{document}